%% file: allart.tex
\documentclass[12pt,reqno]{amsart}
\usepackage{amsaddr}
\usepackage[utf8]{inputenc}
\usepackage{amsmath,amssymb,enumerate,multicol,amsthm, bbm}
\usepackage{mathrsfs,mathabx}
\usepackage{graphicx}
\usepackage{multicol}
\usepackage{enumerate}
\usepackage{color}
\usepackage{hyperref}
\theoremstyle{remark}
\newtheorem{obs}{Remark}

\theoremstyle{plain}
\newtheorem{prop}{Proposition}[section]
\newtheorem{teo}[prop]{Theorem}

\newtheorem{lema}[prop]{Lemma}
\newtheorem{corol}[prop]{Corollary}
\DeclareMathOperator{\Real}{Re}

\DeclareMathOperator{\sgn}{sgn}

\newcommand{\supp}{\operatorname{supp}}
\newcommand{\re}{\mathbb{R}}

\newcommand{\na}{\mathbb{N}}

\newcommand{\normp}[2]{\Vert#1\Vert_{#2}}

\newcommand{\ns}[2]{ \left\| #1 \right\|_{#2}}
\newcommand{\nsc}[2]{ \left\| #1 \right\|^2_{#2}}
\title[On the Cauchy problems associated to a ZK-KP-type family...]{On
  the Cauchy problems associated to a ZK-KP-type family equations with
  a transversal fractional dispersion}
\author{Jorge Morales P.}
\email{jemoralesp@unal.edu.co}
\address{Departamento de Matemáticas, Universidad Nacional de Colombia, Sede--Bogotá}
\author{Félix H. Soriano M.}
\email{fhsorianom@unal.edu.co}
\address{Departamento de Matemáticas,
  Universidad Nacional de Colombia, Sede--Bogotá}
\begin{document}
\maketitle
\begin{abstract}
  In this paper we examine the well-posedness and ill-posedeness of
  the Cauchy problems associated with a family of equations of ZK-KP-type
  \[
    \begin{cases} 
 u_{t}=u_{xxx}-\mathscr{H}D_{x}^{\alpha}u_{yy}+uu_{x}, \cr
 u(0)=\psi \in Z
\end{cases}
\]
in anisotropic Sobolev spaces, where $1\le \alpha \le 1$,
$\mathscr{H}$ is the Hilbert transform and $D_{x}^{\alpha}$ is the
fractional derivative, both with respect to $x$.
\end{abstract}

\input{introart}
\input{prelimart}
\input{bpkatoart}

\input{resumen2art}
\input{Factart}
\bibliographystyle{acm}
\bibliography{../../Fabian/mybiblio}
\end{document}

%% file: introart.tex
\section*{Introduction}
Nonlinear evolution equations play an important role in different
areas of science and engineering. Some of them are worth mentioning:
fluid mechanics, plasma physics, fiber optics, solid state physics,
chemical kinetics, chemical physics and geochemistry, among
others. From the study of their solutions, an attempt is made to
understand the effects of dispersion, diffusion, reaction and
convection associated with the models described by them. For example,
the Korteweg-de Vries (KdV) equation
\begin{equation}\label{KdV}
u_{t}=u_{xxx}+uu_{x} \qquad(x,t)\in \mathbb{R}^2,
\end{equation}
which models the behavior of water waves in shallow channels, has
solitary waves as solutions that behave like particles, which is why
Kruskal and Zabusky called them \emph{solitons} in their 1965 work
(\cite{zbkru}). These solitons are stable, in the sense that if a
solution of the KdV equation (equation \eqref {KdV}) that 
differs very little in shape from soliton-type solutions, at the
beginning, its shape will maintain an aspect that will differ very little from
the shape of a soliton-type solution through time (see \cite{Ben72} and
\cite{Bona75}); in fact, these solutions eventually take the form of 
solitons (see \cite{pewe}). From a practical point of view, the
notion of soliton stability guarantees that, with meticulous care, in
the laboratory we will be able to reproduce these phenomena, first
observed by J. Scott Russell in 1834.\par The Benjamin-Bona-Mahony (BBM)
equation 
\begin{equation}\label{BBM}
u_{t}+u_{x}+uu_{x}-u_{xxt}=0,
\end{equation}
was introduced in \cite{BBM} with the intention to modeling the
propagation of long waves of small amplitude, where the dispersion 
effect is purely nonlinear. The way in which this was obtained, it was
pursued to arrive at an equation equivalent to the equation KdV
\eqref{KdV}. It is interesting to note that despite this intention,
from a purely mathematical point of view, these equations present
significant and interesting differences. \par Other one-dimensional
equations are, one, introduced independently by Benjamin in
\cite{Benjamin} and Ono in \cite{ono},
\begin{equation}\label{BO}
u_{t}+\mathscr{H}u_{xx}+uu_{x}=0.
\end{equation}
which models the internal waves into deep stratified fluids, where
$\mathscr{H}$ is the Hilbert transform.
%%\begin{equation}\label{gKdV}
%%u_{t}+u_{xxx}=f(u)_{x}
%%\end{equation}
The another, is the regularized Benjamin-Ono (rBO)
\begin{equation}\label{rBO}
u_{t}+u_{x}+uu_{x}+\mathscr{H}u_{xt}=0
\end{equation}
where $u=u(x, t)$ is a real function, with $x, t\in \mathbb{R}$. This
equation is a model for wave time evolution with large ridges at the
interface between two immissible fluids. \par
There are two-dimensional versions that extend the equations
previously mentioned. In the case of the KdV equation we have the
Kadomtsev-Petviashvilli (KP) equation, see \cite{KP}, 
\begin{equation}\label{KP}
(u_{t}+auu_{x}+u_{xxx})_{x}\pm u_{yy}=0,
\end{equation}
that describes waves in thin films of high surface tension. Another one is the
Zakharov-Kuznetsov (ZK) equation, see \cite{ZakKuz}, 
\begin{equation}\label{ZK}
u_{t}=(u_{xx}+u_{yy})_{x}+uu_{x},
\end{equation}
which arises in the study of the dynamics of geophysical fluids in
isotropic sets (means in which the characteristics of the bodies do
not depend on direction) and ionic acoustic waves in magnetic
plasmas. 
 
As a two-dimensional extension of the Benjamin-Ono equation we
consider the following family of equations 
\begin{equation}\label{BO2pre}
 \left(u_{t}+u^pu_{x} + \mathscr{H}(u_{xx}+\alpha u_{yy})\right)_x -\gamma u_{yy}=0 \quad p\in\na.
\end{equation}
This is the model of weakly nonlinear dispersive long wave motion in a
two-fluid system, where the interface is capillarized and the bottom
fluid is infinitely deep (see \cite{ablow}, \cite{abloseg} and \cite{kim}). 
Another one version that has received some attention in recent literature
is the ZK-BO equation, 
\begin{equation}\label{ZK-BO}
 u_{t}=(\mathscr Hu_{x}+u_{yy})_{x}+ uu_{x}.
\end{equation}
\par To finish this introduction we will mention the equations that
present a more general type of dispersion that include as particular
cases those mentioned above. A case of these is the integro
differential equation of Whitham
\begin{equation}\label{eq:whitham}
u_t +\alpha uu_x+ k*u_x=0.
\end{equation}
This was introduced by Whitman in \cite{whitham} to model the
breakage of dispersive waves in water. It is clear that if $ k =
\delta + \delta '' $ or $ k = \mathrm {v.p.} \dfrac1x $ we have the 
aforementioned KdV and BO equations. There is also what Lannes and
Saut called the fractional dispersion equation KdV, fKdV (see
\cite{lannessaut}) 
\begin{equation}\label{eq:fKdV}
u_t +\alpha uu_x+ \partial_x D^\alpha u=0,
\end{equation}
where $D=\sqrt{-\partial_x^2}$. In the papers \cite{castro},
\cite{SautKlein}, \cite{linpilsaut1} and \cite{linpilsaut2}
it is established the relationship between the parameter $\alpha $ and the
arising of singularities or the global existence of solutions in
their energy space, as well as how this behavior is related to the
existence and stability of the solitary waves associated with these
equations. It should be said that in this direction there are
interesting open problems, although they can really be very
difficult. \par At this time we also think it is important to mention
the work of Kenig, Martel and Robbiano (\cite{KeMaRo}) that
demonstrates the blow-up of solutions in the energy space of the
generalized scattering equation BO (a slightly different version of
the fKdV)
\begin{equation}\label{eq:fKdVBO}
u_t +|u|^{2\alpha} u_x+ \partial_x D^\alpha u=0,
\end{equation}
when the initial condition is ``larger'' than the ``shape'' of the
solitary wave associated with this equation. \par In the
two-dimensional case we have the equation introduced by Lannes in
\cite{lannesbook} to model long waves of small amplitude in weakly
transverse regimes 
\begin{equation}\label{eq:FDKP}
u_t +\mu \frac32 u u_x+  c_{ww}(\sqrt{\mu} |D^\mu|) \left( 1+\frac
  {D_x^2}{D_y^2} \right)^{\frac12} u=0,
\end{equation} 
where 
$$ c_{ww}( k)=\left((1+\beta k^2) \frac {\tanh(k)}k
\right)^\frac12\ \text{ and} \ |D^\mu|=\sqrt{D_x^2+\mu D_y^2}. 
$$ 
In \cite{lannessaut} the relation between the equations KP and FDKP is
established. In fact, they highlight the difference in the dispersive
character with respect to the parameter $\beta $. Furthermore, by
limiting the parameter $\mu$, the equations KPII or KPI are obtained
when $\beta=0$ or $\beta>0$ respectively. An interesting conjecture
here is to regard the existence of solitary waves for values $\beta $ greater
than $ 1/3 $. \par 
In this paper we consider the ZK-KP type Cauchy problem
\begin{equation}
  \label{eq:principal}
\begin{cases} 
 u_{t}=u_{xxx}-\mathscr{H}D_{x}^{\alpha}u_{yy}+uu_{x}, \cr
 u(0)=\psi \in Z
\end{cases}
\end{equation} 
for $-1\le \alpha \le 1$, where $\mathscr{H}$ denotes the Hilbert
transform in the $x$ variable defined by
\[\mathscr{H}(f)=\mathrm{p. v. }\frac1\pi\int \frac{f(y)}{x-y}\, dy
  \quad f \in H^{s}(\mathbb{R}),\] for each $f\in \mathcal S$,
$D_{x}^\alpha$ is the homogeneous fractional derivative in
$x$ variable defined by
\[ \widehat {D_{x}^{\alpha}f } (\xi,\eta)=
  |\xi|^{\alpha}\hat{f}(\xi,\eta),\] and $Z$ is one of the
Sobolev spaces $X^{s_1, s_2}$, $\widehat X^{s_1, s_2}$, $Y^{s_1, s_2}$ and
$\widehat Y^{s_1, s_2}$, that  we will specify in the notations.
\par The cases $\alpha=1$ and $\alpha=-1$ are the Cauchy problems
corresponding to the very popular ZK and KPI equations, respectively,
that were mentioned earlier. If $u(x,y,t)$ is a solution to
\eqref{eq:principal}, then $u_\lambda$ given by
$$u_\lambda (x,y,t)=\lambda^2 u(\lambda x,\lambda^{\frac{3-\alpha}{2}}
y ,\lambda^3 t)$$ is also a solution to (\ref{eq:principal}) and
\begin{equation*}
	\left\|u_\lambda(t)\right\|_{ \dot{H}^ {s_1,s_2}
          (\re^2)} =\lambda^{s_1+\frac{3+\alpha}{4}} \left \| u(t)
        \right \|_{ \dot{H}^ {s_1, s_2}( \re^2)}
\end{equation*}
with $s_2=\frac{2s_1}{3-\alpha}$. This suggests that the local
well-posedness could be guaranteed in
$H^{s_1,\frac{2s_1}{3-\alpha}}(\re^2)$ for
$s_1\geq -\frac{3+\alpha}{4}$.\par
In this work we propose to show the local well-posedness of the Cauchy
problem \eqref{eq:principal} in the Sobolev spaces $Z$ mentioned
above. For this purpose, we use Kato's theory for quasilinear
equations and the ideas introduced by Kenig \cite{kenig2004} for the
KP-I equation and developed by Linares, Pilod and Saut in
\cite{linpilsaut3} for the f-KPI and f-KPII equations. Specifically,
it is making the use of the Strichartz estimate provided by the group
generated by the homogeneous linear equation associated with
\eqref{eq:principal}, making use of energy estimates. We will also
make some ill-posedness observations of this equation for
$-1 \le\alpha <0 $. For this we will use the ideas developed by
Molinet, Saut and Tzvetkov in \cite{MST}. More precisely, we will show
that the flow associated with the solutions of \eqref{eq:principal} is
not of class $C^2$. This, in particular, implies that it cannot be
applied the Picard iteration method to get a solution to the integral
equation obtained by the Duhamel principle applied to \eqref{eq:principal}.

%%% Local Variables:
%%% mode: latex
%%% TeX-master: "allart"
%%% End:

%% file: prelimart.tex
%%% Local Variables:
%%% mode: latex
%%% TeX-master: "allart"
%%% End:

\section*{Notation}

\begin{enumerate}
\item $\mathcal{S}(\re^2)= \mathcal{S}$ denotes  the Schwartz space and
  $\mathcal{S}'(\re^2)=\mathcal{S}'$ denotes its topological dual vector
  space, the tempered distributions.
  \item $H^s(\re^2)=H^s$ is the $s^{th}$-order Sobolev space.
  \item For a variable or an operator $u$, we denote by $\langle u
    \rangle$ the expression $(1+u^2)^{\frac12}$. 
\item For $s_1,s_2\in\re $, the anisotropic Sobolev space 
  $H^{s_1,s_2}(\re^2)$ is defined by
 $$H^{s_1,s_2}(\re^2)=\left\{f\in \mathcal{S}' \;\left |
     \; \int_{\re^2}{(\langle \xi \rangle^{2s_1}+ \langle \eta\rangle^{2s_2})
       |\widehat{f}(\xi,\eta)|^2 d\xi d\eta } <\infty
   \right. \right\}.$$ The norm in this space is given by
 $$ \|f\|_{H^{s_1,s_2}(\re^2)} =\sqrt{
   \int_{\re^2}(\langle \xi \rangle^{2s_1}+ \langle
   \eta\rangle^{2s_2}) |\widehat{f}(\xi,\eta)|^2 d\xi d\eta}, $$ for
 all $f$ in this space. When there is no risk of confusion, we denote
 this space by $H^{s_1,s_2}$. Observe that $H^{s,s}=H^s$, and the
 immediately above norm is equivalent to the usually given in the
 literature.
   \item $D_{x}^s$, $D_{y}^s$, $J_{x}^s$, $J_{y}^s$ and
     $J^s$ denotes the operators defined, via Fourier transform,
     by
\begin{equation*}
\begin{split}
\widehat{D_{x}^sf}&=|\xi|^s \widehat{f},\\
\widehat{D_{y}^sf}&=|\eta|^s \widehat{f},\\
\widehat{J_{x}^sf}&=(1+|\xi|^2)^{\frac{s}{2}} \widehat{f},\\
\widehat{J_{y}^sf}&=(1+|\eta|^2)^{\frac{s}{2}} \widehat{f} \quad \text{and}\\
\widehat{J^sf}&=(1+|\xi|^2+|\eta|^2)^{\frac{s}{2}} \widehat{f},
\end{split}
\end{equation*}
for any $f\in \mathcal S'(\re^2)$. 
\item For $s_1,s_2\geq 0$, we denote by $X^{s_1,s_2}(\re^2)=
  X^{s_1,s_2}$ the space $$X^{s_1,s_2}(\re^2)=\{ f\in H^{s_1,s_2}\;
  |\; \partial_x^{-1} f \in H^{s_1,s_2} \}.$$ The norm in
  this space is given by 
  \begin{equation*}
\|f\|_{X^{s_1,s_2}}^2= \|f\|_{H^{s_1,s_2}}^2 +\| \partial^{-1}_x  f\|^2_{H^{s_1,s_2}}.
\end{equation*}
\item For $s_1,s_2\geq 0$,  we denote by $\widehat X^{s_1,s_2}(\re^2)=
  \widehat X^{s_1,s_2}$ the  space $$\widehat X^{s_1,s_2}(\re^2)=\{ f\in H^{s_1,s_2}\;
  |\; \partial_x^{-1} f \in L^2 \}.$$ The norm in
  this space is
  \begin{equation*}
\|f\|^2_{\widehat X^{s_1,s_2}}= \|f\|^2_{H^{s_1,s_2}} +\| \partial^{-1}_x  f\|^2_{L^2}.
\end{equation*}
\item For $s_1,s_2\geq 0$,  we denote by $X_\alpha ^{s_1,s_2}(\re^2)=
  X_\alpha^{s_1,s_2}$ the  space $$X_\alpha^{s_1,s_2}(\re^2)=\{ f\in H^{s_1,s_2}\;
  |\; \partial_x^{\frac{ \alpha -1} 2} f \in L^{2} \}.$$ The norm in
  this space is given by
  \begin{equation*}
\|f\|_{X^{s_1,s_2}}^2= \|f\|_{H^{s_1,s_2}}^2 +\| \partial^{-1}_x  f\|^2_{H^{s_1,s_2}}.
\end{equation*}
\item For $s_1,s_2\geq 0$,  we denote by $Y^{s_1,s_2}(\re^2)=
  Y^{s_1,s_2}$ the  space $$Y^{s_1,s_2}(\re^2)=\{ f\in H^{s_1,s_2}\;
  |\; \partial_x^{-1} \partial_y f \in H^{s_1,s_2} \}.$$ The norm in
  this space is 
  \begin{equation*}
\|f\|^2_{Y^{s_1,s_2}}= \|f\|^2_{H^{s_1,s_2}} +\| \partial^{-1}_x \partial_y f\|^2_{H^{s_1,s_2}}
\end{equation*}
\item For $s_1,s_2\geq 0$,  we denote by $\widehat Y^{s_1,s_2}(\re^2)=
  \widehat Y^{s_1,s_2}$ the  space $$\widehat Y^{s_1,s_2}(\re^2)=\{ f\in H^{s_1,s_2}\;
  |\; \partial_x^{-1} \partial_y f \in L^2 \}.$$ The norm in
  this space is given by 
  \begin{equation*}
\|f\|_{\widehat Y^{s_1,s_2}}^2= \|f\|_{H^{s_1,s_2}}^2 +\| \partial^{-1}_x \partial_y  f\|^2_{L^2}.
\end{equation*}
\item  We define $H^\infty(\re^2)=\bigcap_{s_1,s_2\geq
    0}H^{s_1,s_2}(\re^2)$. Analogously $X^\infty$, $\widehat X^\infty$,
  $Y^\infty$ and $\widehat Y^\infty$. 
\end{enumerate}

\section{Preliminaries}
We start this section of preliminaries by observing that 
the spaces $X^{s_1, s_2}$, $\widehat X^{s_1, s_2}$, $X_\alpha^{s_1, s_2}$,
$Y^{s_1, s_2}$ and $\widehat Y^{s_1, s_2}$ are Hilbert spaces. Thanks to
the next lemma, each of these spaces is dense in $L^2$
\begin{lema}
  The space $\partial_x \mathcal S$ is dense in $L^2$. In general, it
  is dense also in $H^{s_1, s_2}$, $X^{s_1, s_2}$, $\widehat X^{s_1,
    s_2}$, $X_\alpha^{s_1, s_2}$, $Y^{s_1, s_2}$ and $\widehat Y^{s_1, s_2}$.
\end{lema}
\begin{proof}
  Take a non negative function $\phi\in C^\infty$ defined on the real
  numbers, identically zero on the interval $[-1/2, 1/2]$ and
  identically $1$ out of $[-1, 1]$. For any $\psi\in \mathcal S$ we
  define $\psi_\lambda$ using the equation
  $$\widehat  \psi_\lambda(\xi, \eta) = \phi(\lambda \xi)\widehat \psi(\xi,
  \eta),$$
  for all $(\xi, \eta)\in \re^2$. The Plancherel theorem allows us show
  that $\psi_\lambda$ converges to $\psi$ in $L^2$, as $\lambda \to
  \infty$. The same argument allows us to show that $\partial_x \mathcal
  S$ is dense in any of the spaces mentioned in the lemma statement.
\end{proof}
By a duality argument, we can conclude that $L^2$ is densely contained
in the dual spaces of any of the spaces mentioned in the lemma.\par
As consequence of the previously discussed, we can extend the operator
$\partial^3_{x}-\mathscr{H}D_{x}^{\alpha}\partial^2_{y}$ to the entire
$L^2$, with image in the $X^3$ dual, $(X^3)^*$. Indeed,
$-\partial^3_{x}+\mathscr{H}D_{x}^{\alpha}\partial^2_{y}$ is bounded
from $X^3$ to $L^2$. So its adjoint operator is bounded from $L^2$ to
$(X^3)^*$. Thanks to the Fourier transform, it can be seen that this
adjoint operator is an extension of
$\partial^3_{x}-\mathscr{H}D_{x}^{\alpha}\partial^2_{y}$ to all
$L^2$.\par  As a corollary of above, we have the following lemma.
\begin{lema}
  Let $W_\alpha$ be the unitary group of operators generated by the operator
  $\partial^3_{x}-\mathscr{H}D_{x}^{\alpha}\partial^2_{yy}$ and let $f$
  and $u$ be  continuous functions from an open interval $I$ to
  $L^2$. Then $$ u=W_\alpha(t) \psi +\int_0^{t} 
  W_\alpha(t -t') f(t') \, dt'$$   if, and only if $u$ has continuous derivative
  on $I$ with values in $(X^3)^*$, the dual space $X^3$, and
  \begin{equation}
    \label{eq:1a}
  \partial_t u=
  \partial^3_{x}u-\mathscr{H}D_{x}^{\alpha}\partial^2_{y}u +f.
  \end{equation}
\end{lema}
This lemma gives sense to the well-posedness results in
$H^{s_1, s_2}$, that we shall enunciate later, when $\alpha$ is
negative. \par The next result is a technical lemma that we use later.
\begin{lema}\label{justpin}
  Let us assume that $u$ and $f$ are continuous functions in the $L^2$
  space and satisfy \eqref{eq:1a}  in the sense described there.
  Then,
  \begin{equation}
    \label{eq:1b}
\frac12 \frac d {dt}\|u \|^2 = (u,f).
  \end{equation}
\end{lema}
\begin{proof}
  Assume that $u_\lambda = e^{\lambda(\bigtriangleup - D_x^{-1})}(u)$ 
  and $f_\lambda$ is defined in the same way. $u_\lambda$ and $f_\lambda$
  are in $X^\infty$ and they are uniformly convergent, as $\lambda$ tends to
   0, on closed intervals to $u$ and $f$ in $L^2$. Also, they satisfy
  the equation \eqref{eq:1a}. From the antisimmetry of the operator
  $\partial^3_{x}-\mathscr{H}D_{x}^{\alpha}\partial^2_{y} $, it is
  easy to see that $u_\lambda$ and $f_\lambda$ satisfy
  \eqref{eq:1b}. When we make $\lambda $ tend to 0, we get the lemma.
\end{proof}
Next we shall state a set of results about the properties of the
spaces with we work in this paper. Maybe one the most known of these
results is the Sobolev lemma. Here we present a version for the
$H^{s_1,s_2}(\re^2)$ spaces.
\begin{lema}[Sobolev]\label{tfabi}
Let $s_1$ and $s_2$ be  positive real numbers such that 
$\frac{1}{s_1} + \frac{1}{s_2} < 2.$ Then, 
$H^{s_1,s_2}(\re^2) \subset C_{\infty}(\re^2)$ (the set of 
continuous functions on $\re^2$ vanishing at infinity), with 
continuous embedding.
\end{lema}
\begin{proof}
See \cite{tddfabian}
\end{proof}
\begin{lema}{\label{interenlp}}
Let $1\leq p<q\leq \infty$ and $f\in L^p\cap L^q$. Then $f\in L^r$ for
$r=\theta p +(1-\theta)q$, where $\theta\in (0,1)$, and we have 
\begin{equation*}
\left\|f\right\|_{L^{r}}^{r}\leq \left\|f\right\|_{L^{p}}^{\theta p}
\left\|f\right\|_{L^{q}}^{(1-\theta) q} 
\end{equation*}  
\begin{proof}
The proof is immediate consequence of the Hölder inequality. 
\end{proof}
\end{lema}
\begin{lema}\label{lpenhs}
  If $s \in (0,n/2)$, then $H^s(\re^n)$ is a continuous embedding in
  $L^p(\re^n)$, with $p = \frac{2n}{n-2s}$, i.e.
  $s=n(\frac{1}{2}-\frac{1}{p})$. Furthermore, for $f\in H^s(\re^n)$,
  $s\in(0,n/2)$
\begin{equation*}\label{eclpenhs}
\left\|f\right\|_{L^p(\re^n)} \leq c_{n,s}\left\|D^s f\right\|_{L^2(\re^n)} \leq c\left\|f\right\|_{H^s(\re^n)}
\end{equation*}
where $$D^l f = (-\Delta)^{l/2}=((|\xi|)^l \hat f)^{\vee}$$
\end{lema}
\begin{proof}
See the Linares and Ponce book \cite{linaresponce} page 48.
\end{proof}

\begin{lema}\label{interdsenlplq}
  Let $s_1,s_2\in\re$ and assume that $D^{s_1}f \in L^p(\re)$ and
  $D^{s_2}f \in L^q(\re)$. Then, for all $\theta\in[0,1]$,
  $D^sf\in L^r$, and
  $$\|D^{s}f\|_{L^r(\re)} \leq C_s
  \|D^{s_1}f \|_{ L^p(\re)}^{ \theta} \| D^{s_2}f\|_{L^q(\re)}^{1-\theta} $$ 
  where $\theta = \dfrac{s_2-s}{s_2-s_1}$ and
  $\dfrac{1}{r} = \dfrac{\theta}{p}+\dfrac{1-\theta}{q}$.
\end{lema}
\begin{proof}
  A proof can be found in \cite{tesjul} 
\end{proof}
The next estimate was proved by Kato and Ponce in
\cite{katoponce} and it will be very useful later in this work.
\begin{lema}\label{katoponce}
For $s>0$ and $1<p<\infty$, we have 
\begin{equation}
  \|[J^s,f ]g \|_{L^p(\re^n)}\lesssim \|\partial f\|_{L^{\infty}(\re^n
    )}\|J^{s-1}g\|_{L^{p}(\re^n )}+\|J^sf\|_{L^{p}(\re^n
    )}\|g\|_{L^{\infty}(\re ^n)},
\end{equation}
for all $f$ and $g\in \mathcal S(\re^n )$ 
\end{lema}
\begin{corol}\label{coro-katopo}
For $s>0$ and $p\in(1\, \infty)$, $L_s^p \cap L^{\infty}$ is an 
algebra, also
\begin{equation}
\|fg\|_{L_s^p} \leq c(\|f\|_{\infty}\|g\|_{L_s^p}+ \|f \|_{L_s^p} \|g \|_{\infty})
\end{equation}
\end{corol}
The  following estimate for the commutator operator can be seen in
\cite{linaresponce} page 51.
\begin{lema}\label{lemaconmkatoponce2}
For $s>0$, we have
\begin{equation}\label{katoponce2}
\left\|\left[\partial_x^s,g\right]f\right\|_{L^2(\re)}\lesssim \left\|\partial_xg\right\|_{L^{\infty}(\re)}\left\|\partial_x^{s-1}f\right\|_{L^{2}(\re)}+\left\|\partial_x^sg\right\|_{L^{2}(\re)}\left\|f\right\|_{L^{\infty}(\re)}
\end{equation}
\end{lema}
The next result is the Leibniz rule for fractional derivatives and it was
proved by Kenig, Ponce and Vega in \cite{KPV1993}.
\begin{lema}\label{reglaleibniz}
For $\alpha \in(0,1)$, we have
\begin{equation}\label{eqreglaleibniz}
\left\|D_x^\alpha (fg)\right\|_{L^p(\re)}\leq
\left\|D_x^\alpha(f)\right\|_{L^{p_1}(\re)}\left\|g\right\|_{L^{q_1}(\re)}+\left\|
  D_x^\alpha g\right\|_{L^{p_2}(\re)}\left\|f\right\|_{L^{q_2}(\re)} 
\end{equation}
where $1<p_1, p_2, q_1,q_2\leq \infty$ and satisfy $\frac{1}{p_1 }+\frac{
  1}{q_1}=\frac{ 1}{ p_2}+ \frac{1}{q_2 }=\frac{1}{p}$.
\end{lema}
\begin{lema}\label{xenlp}
Let $-1\le \alpha\le 1$ and $0\leq p \leq \frac{8}{1-\alpha}$ 
 \begin{equation*}
\|f\|_{L^{p+2}(\re^2)}^{p+2} \lesssim
\|f\|_{L^{2}(\re^2)}^{2-\frac{p(1-\alpha)}{4}}
\|\partial_xf\|_{L^{2}(\re^2)}^{\frac{p (3-\alpha)
  }{4}}\|D^{\frac{\alpha-1}{2}}_x\partial_yf\|_{L^{2}(\re^2)}^{\frac p2}
\end{equation*}
\begin{proof}
  First, let us prove the lemma for $p=p*=\frac{8}{1-\alpha}$. From
  Lemmas \ref{lpenhs} and \ref{interdsenlplq} we have
\begin{equation}\label{inmerp1}
\begin{split}
	\left\|f\right\|_{L^{p^*+2}(\re^2)}^{p^*+2}&=\int_{\re^2}{\left|f(x,y)\right|^{p^*+2}}dxdy\\
	&=\int_{\re}{\left\|f(\cdot,y)\right\|_{L^{p^*+2}_x}^{p^*+2}}dy\\ 
	&\leq C \int_{\re}{\left\|D^{\frac{p^*}{p^*+2}}_xf(\cdot,y)\right\|_{L^{2}_x}^{p^*+2}}dy\\ 
	&\leq C\int_{\re}{\left\|\partial_xf(\cdot,y)\right\|_{L^{2}_x}^{2}\left\|D^{\frac{p^*-4}{2p^*}}_xf(\cdot,y)\right\|_{L^{2}_x}^{p^*}}dy\\
	&\leq C \left\|\partial_xf\right\|_{L^{2}(\re^2)}^{2} \sup_{y\in \re}{\left\|D^{\frac{p^*-4}{2p^*}}_xf(x,y)\right\|_{L^{2}_x}^{p^*}}.
\end{split}
\end{equation}
On the other hand, for all $y\in \re$,
 \begin{equation}\label{inmerp2}
 \begin{split}
   \left\|D^{\frac{p^*-4}{2p^*}}_xf(\cdot,y)\right\|_{L^{2}_x}^{2}	&=\int_{\re}{\left|D_x^{\frac{p^*-4}{2p^*}}f(x,y)\right|^{2}}dx\\
   &= 2\int_{\re}{\int_{-\infty}^y{D_x^{\frac{1+\alpha}{4}}f(x,\eta)D_x^{\frac{1+\alpha}{4}}\partial_yf(x,\eta)}}d\eta dx\\
   &= 2\int_{-\infty}^y{\int_{\re}{D_x
       f(x,\eta)D_x^{\frac{\alpha -1}2}}\partial_yf(x,\eta)}dxd\eta\\
   &\leq 2\int_{-\infty}^y{\left\|\partial_xf(\cdot,\eta)\right\|_{L^{2}_x}\left\|D^{\frac{\alpha -1}2}_x\partial_yf(\cdot,\eta)\right\|_{L^{2}_x}}d\eta\\
   &\leq 2
   \left\|\partial_xf\right\|_{L^{2}(\re^2)}\left\|D^{\frac{\alpha
         -1}2}_x\partial_yf\right\|_{L^{2}(\re^2)}.
 \end{split}
\end{equation}
Therefore,
\begin{equation}\label{inmerp3}
\|f\|_{L^{p^* +2}(\re^2)}^{p^* +2} \leq c
\|\partial_xf\|_{L^{2}(\re^2)}^{\frac{p^*(3-\alpha)}4}
\|D^{\frac{\alpha -1} 2}_x\partial_yf\|_{L^{2}(\re^2)}^{\frac
{p^*}2}	
\end{equation}
The inequality for $0<p<p^*$ follows immediately from Lemma
\ref{interenlp} and the last inequality.
\end{proof}
\end{lema}
\subsection{Kato's theory}
We will make a brief presentation of Kato's theory described in
\cite{katoLN448}.  With this it can be showed the well-posedness of
the Cauchy problems associated to linear and quasilinear evolution equations.

\subsubsection{Linear case}

Suppose that  $X$ and $Y$ are reflexives Banach spaces with
$Y\subseteq{X}$ in a dense and continuous way, and let
$\{A(t)\}_{t\in[0,T]}$ be a operators family such that
\begin{enumerate}
\item $A(t)\in{G}(X,1,\beta)$. In other words, $-A(t)$ generates a
$C_0$-semigroup such that
\[\|e^{-sA(t)}\|\leq{e}^{\beta{s}},\]  for all \(s\in[0,\infty).\)

\item There exists an isomorphism $S:Y\to {X}$ such that
$SA(t)S^{-1}=A(t)+B(t)$, where $B(t)\in{B}(X),$ for
$0\leq{t}\leq{T},$ $t\to {B}(t)x$ is strongly measurable, for
each $x\in{X}$, and $t\to \|B(t)\|_X$ is integrable in $[0,T]$.

\item $Y\subseteq{D}(A(t)),$ for $0\leq{t}\leq{T}$, and
  $t\rightarrow{A}(t)$ is strongly continuous from $[0,T]$ to $B(Y,X)$.

\end{enumerate}

\begin{teo}\label{tkatolin}
Under the above conditions, there exists a operators family
$\{U(t,s)\}_{0\leq{s}\leq{t}\leq{T}}$ such that:

\begin{enumerate}
\item $U$ is strongly continuous from $\Delta \to {B}(X)$,
where $\Delta=\{(t,s):0\leq{s}\leq{t}\leq{T}\}$.

\item $U(t,s)U(s,r)=U(t,r)$ for $(t,s)$ and $(s,r)\in\Delta$, and
$U(s,s)=I$.

\item $U(t,s)Y\subset{Y}$ and  $U$ is strongly continuous from
$\Delta\to {B}(Y)$.

\item $\dfrac{dU(t,s)}{dt}=-A(t)U(t,s),$
  $\dfrac{dU(t,s)}{ds}=U(t,s)A(s),$ in the strong sense in
  $B(X,Y)$ space and are strongly continuous from $\Delta\to
  {B}(X,Y)$.
\end{enumerate}
\end{teo}
The operators family $\{U(t,s)\}_{0\leq{s}\leq{t}\leq{T}}$ in the
previous theorem is called \emph{the evolution operators} associated
to ${A(t)}$. An immediate consequence from the last theorem is, for
$\varphi \in{Y}$, $u(t)=U(t,s)\varphi $ is solution to the Cauchy problem
\begin{align*}
&\frac{du}{dt}+A(t)u=0 \qquad \text{ for } \quad s\leq{t}\leq{T},\\
&u(s)=\varphi.
\end{align*}
Moreover, if $f\in{C}([0,T];X)\cap{L^1}([0,T];Y)$, then
\[u(t)=U(t,0)\varphi+\int_0^tU(t,s)f(s)ds\]
if and only if $u\in{C}([0,T];Y)\cap{C}^1((0,T);X)$ and
\begin{align*}
\frac{du}{dt}+A(t)u&=f(t) \quad\text{ for } \quad 0\leq{t}\leq{T},\\
u(0)&=\varphi.
\end{align*}

\subsubsection{Quasilinear Case}\label{secttkato}
Let $X$ and $Y$ be reflexives Banach spaces, $Y\subseteq X$, with
dense and continuous embedding. Let us consider the following  problem
\begin{equation}\label{Q}
\begin{array}{ll}
&\partial_t u +A(t,u)u=f(t,u)\in X, \ 0< t , \\
&u(0)=u_0\in Y,
\end{array}
\end{equation}
where, for each $t$, $A(t,u)$ is a linear operator from $Y$
to $X$ and $f(t,u)$ is a function from $\re \times Y$ in $X$.
Let us also consider the next conditions:
\par
$(X)$ There exists an isometric isomorphism $S$ from $Y$ to $X$. \\
There exist $T_0>0$ and $W$ an open ball with $w_0$ as center such that: \par
$(A_1)$ For each $(t,y)\in [0,T_0]\times W$, the linear operator
$A(t,y)$ belongs to $G(X,1,\beta)$, where $\beta$ is a positive
real number. In the other words, $-A(t,y)$ generate a $C_0$
semigroup such that
\[
\normp{e^{-sA(t,y)}}{\mathcal B(X)}\le e^{\beta s}, \ \text{for}\
s\in[0,\infty).
\]
Note that if $X$ is a Hilbert space, $A \in G(X,1,\beta)$ if, and only if,
\begin{enumerate}
\item[a)] $\langle Ay, y\rangle_{X}\ge -\beta\normp{y}{X}^2$ for all
$y\in D(A)$,
\item[b)] $(A+\lambda)$ is onto for all $\lambda >\beta.$
\end{enumerate} (See \cite{katoPT} or \cite{reed}) \par
$(A_2)$ For each $(t,y)\in[0,T_0]\times W$ the operator
$B(t,y)=[S,A(t,y)]S^{-1}\in \mathcal B(X)$ and is uniformly bounded,
i.e., there exists $\lambda_1>0$ such that
\begin{align*}
& \normp{B(t,y)}{\mathcal B(X)}\le\lambda_1\ \ \text{for all}\
(t,y)\in[0,T_0]\times W, \\ \intertext{Furthermore, for some $\mu_1>0$,
    we have that, for all $y$ and $z\in W$,}
&\normp{B(t,y)-B(t,z)}{\mathcal B(X)}\le\mu_1\normp{y-z}{Y}.
\end{align*}\par
$(A_3)$ $Y\subseteq D(A(t,y))$, for each $(t,y)\in[0,T_0]\times W,$
(the restriction of $A(t,y)$ to $Y$ belongs to $\mathcal B(Y,X)$)
and, for each $y \in W$ fix, $t\to A(t,y)$ is strongly
continuous. Besides, for all $t\in[0,T_0]$ fix, it is satisfied the
following Lipschitz condition,
\[
\normp{A(t,y)-A(t,z)}{\mathcal B(Y,X)}\le\mu_2\normp{y-z}{X},
\]
where $\mu_2\ge 0$ is constant.

$(A_4)$ $A(t,y)w_0\in Y$ for all $(t,y)\in[0,T]\times
W$. Also, there exists a constant $\lambda_2$ such that
\[
\normp{A(t,y)w_0}{Y}\le\lambda_2, \ \text{for all}\
(t,y)\in[0,T_0]\times W.
\]\par
$(f_1)$ $f$ is a bounded function  in $[0,T_0]\times W$ to $Y$, i.e., there exists $\lambda_3$ such that
\[
\normp{f(t,y)}{Y} \le\lambda_3, \ \text{for all}\ (t,y)\in[0,T_0] \times W,
\] Also, the function $t\in [0,T_0] \mapsto f(t,y)\in Y$ is continuous
with respect to the topology of $X$ and for all $y$ and $z\in Y$ we
have that 
\begin{align*}
&\normp{f(t,y)-f(t,z)}{X}\le\mu_3\normp{y-z}{X}, \end{align*} where
$\mu_3\ge 0$ is a constant.
\begin{teo}[Kato]\label{tkato}
Assume that the conditions $(X),$ $(A_1)-(A_4)$ and $(f_1)$ are
satisfied. Given $u_0\in Y$, there exist $0<T<T_0$ and a unique $u\in
C([0,T];Y)\cap C^1((0,T);X)$ solution to \eqref{Q}.  Furthermore, the
map $u_0\to u$ is continuous in the following sense:
consider the sequence of Cauchy problems,
\begin{equation}
\begin{aligned}
\label{dependencia}
&\partial_t u_n+A_n(t,u_n)u_n=f_n(t,u_n)\ \ t>0 \\
&u_n(0)=u_{n_0} \ n\in \mathbb N.
\end{aligned}
\end{equation}
Suppose that the conditions $(X)$, $(A_1)$--$(A_4)$ and $(f_1)$ 
are also satisfied for all $n\ge 0$ in \eqref{dependencia}, with the same
$X, \ Y$ and $S$, and the correspondents $\beta$,
$\lambda_1$--$\lambda_3$, $\mu_2$--$\mu_3$ can be chosen
independent of $n$. Also, let us suppose that
\begin{align*}
\mathop{s\text{-}\lim}_{n\to\infty} A_n(t,w)&=A(t,w) \ \text{in } \ B(X,Y),\\
\mathop{s\text{-}\lim}_{n\to\infty} B_n(t,w)&=B(t,w) \ \text{in } \ B(X),\\
\lim_{n\to\infty} f_n(t,w)&=f(t,w) \ \text{in } \ Y,\\
\lim_{n\to\infty} u_{n_0}&=u_0  \ \text{in } \ Y,
\end{align*}
where $s$-$\lim$ denotes the strong limit. Then, $T$ can be taken
in such a way that $u_n\in C([0,T],Y)\cap C^1((0,T),X)$ and
\[
\lim_{n\to\infty} \sup_{[0,T]}\normp{u_n(t)-u(t)}{Y}=0.
\]
\end{teo}
A proof of this theorem can be found in \cite{katoLN448} and
\cite{kobayasi}.
\subsection{Other results}
\begin{prop}[Kato's inequality]\label{descato}
Let $f\in H^s$, $s>2$, $\Lambda=(1-\Delta)^{1/2}$ and $M_f$ be the
multiplication operator by $f$. Then, for $|\tilde{t}|, |\tilde{s}|\leq s-1$, 
$\Lambda^{-\tilde{s}}[\Lambda^{\tilde{s}+\tilde{t}+1},M_f]\Lambda^{-\tilde{t}} 
\in B(L^2(\re^2 ) ) $ and 
\begin{equation}
 \left\| \Lambda^{-\tilde{s}} [\Lambda^{\tilde{s}+\tilde{t}+1},M_f] 
\Lambda^{-\tilde{t}} \right\|_{{B} \left(L^2\left(\re^2\right) 
\right)} \leq c\left\| \nabla f \right\|_{H^{s-1}}.
\end{equation}
\end{prop}
\begin{prop}\label{gengrup}
Let $f:\re^2\to \re$ be a bounded continuous function such that $\partial_x f$ 
exists and  is continuous and bounded. Then, if $A=f\partial_x$, 
\begin{equation}\label{desigualdad120}
\langle A(u),u\rangle_{L^2}\geq 
-\frac{1}{2}||\partial_xf||_{L^{\infty}}||u||_{L^2}^2,
\end{equation}
for all $u\in D(A)$, $A+\lambda$ is onto, for all $\lambda > \frac 12 
||f||_{L^{\infty}}$. In particular, $A \in 
G\left(L^2\left(\re^2\right),1,\frac12||f||_{L^{\infty}}\right)$. 
\end{prop}
\begin{proof}
The inequality \eqref{desigualdad120} is obtained immediately after using 
the integration by parts. Let us see that $A+\lambda$ is onto, if $\lambda > 
\frac12||f||_{L^{\infty}}$. Suppose that $\psi$ is such that $\langle 
(A+\lambda)(u),u\rangle_{L^2}=0$, for all $u\in D(A)$. Then $\psi 
\in D(A^*)\subseteq D(A)$. From  \eqref{desigualdad120}, it follows that 
$$0\geq \langle (a+\lambda)(u),u\rangle_{L^2}\geq (\lambda 
-\frac12||f||_{L^{\infty}})||\psi||_{L^2}^2.$$ Hence, $\psi=0$ and, therefore, 
$A+\lambda$ is onto.
\end{proof}

%%% Local Variables:
%%% mode: latex
%%% TeX-master: "allart"
%%% End:

%% file: bpkatoart.tex
%%% Local Variables:
%%% mode: latex
%%% TeX-master: "allart"
%%% End:
\section{Local well-posedness in Sobolev spaces of $s^{th}$ order with
$s> 2$}
In this section we examine the local well-posedness of the problem
\eqref{eq:principal} in the  Sobolev spaces $H^{s}$,
$X^{s}$, $\widehat X_\alpha^{s}$, $Y^{s}$ and $\widehat Y^{s}$, for $s>2$.
\subsection{ Local well-posedness in $H^{s}(\re^2)$}
In this section, we will make use of Kato's theory to show the local
well-posedness of \eqref{eq:principal} in the $H^{s}$ spaces. More
precisely we have the following theorem.
\begin{teo}\label{principal}
  Let $s$ and $\alpha$ be real numbers such that $s>2$ and
  $-1\le \alpha\le 1$.  For $\psi \in H^{s}(\re^2)$, there exist
  $T>0$, that depends only on $\|\psi\|_{H^{s }}$, and a unique
  $u\in C( [0,T], H^{s}(\re^2) )\cap$
  $ C^1( [0,T], H^{s-3}(\re^2 )\cap  (X^3)^* )$ solution to the
  Cauchy problem \eqref{eq:principal}\par
  Moreover, the map $\psi \to u$ from $H^{s}$ to
  $C( [0,T], H^{s}(\re^2 ) )$ is continuous.
\end{teo}
\begin{proof}
  Let $W_\alpha(t)$ be the operators group defined by
  $$W_\alpha(t)\psi = e^{t(\partial_x^3-\mathscr{H}
    \partial_y^2)}\psi= \left(e^{-it (\xi^3  +
    \sgn(\xi)|\xi |^\alpha \eta^2)}\widehat{ \psi
  }\right)^{\vee}, $$ for all
$\psi\in H^{s}$. $u$ is solution to the 
problem \eqref{eq:principal} if and only if $v=W_\alpha(t)u$ is
solution to the problem
\begin{equation}\label{problemaeq}
\begin{cases}
 v_t+A(t,v)v=0\\
 v(0)=\psi,
\end{cases} 
\end{equation}
where $A(t,v)= W_\alpha(t)(W_\alpha(-t)v)\partial_x W_\alpha(-t)$. Let
us see that this last problem satisfies each condition of Kato's
theorem (Theorem \ref{tkato}).

Let $X=L^2(\re^2)$, $Y=H^{s}(\re^2)$ and $S=\Lambda_{s}=J^{s}$. From
Plancherel's theorem, it is evident that $S$ is an isomorphism between
$X$ and $Y$.

With the following lemmas we show that the conditions $(A_1)$-$(A_4)$
are satisfied. 
\begin{lema}
 $A(t,v)\in G(X,1,\beta(v))$, where $\beta(v)= \frac12 
\sup\limits_t\|\partial_xW_\alpha(t)v \|_{ L^ { 
\infty } } $
\end{lema}
\begin{proof}
Since $\{W_\alpha(-t)\}$ is an strongly continuous unitary operators  
group and $u \in H^{s}(\re^2)$, from Proposition  \ref{gengrup},
we obtain the result.
\end{proof}

\begin{lema}
  For $S$ given as above, $$SA(t,v)S^{-1}=A(t,v)+B(t,v),$$ where
  $B(t,v)$ is a bounded operator in $L^2$, for all  $t\in \re$ and
  all $v\in H^{s}$, and  satisfies the inequalities
\begin{align}
||B(t,v)||_{\mathcal{B}(L^2)}&\leq \lambda(v) \label{2.3}\\
||B(t,v)-B(t,v')||_{\mathcal{B}(L^2)}&\leq \mu||v-v'||_{H^{s}}\label{2.4} 
\end{align}
for $t\in \re$ and all $v$ and $v'\in H^{s}( \re^2)$. Where
$\mu$ is a  positive real number and 
\begin{equation*}
  \lambda(v)= \sup\limits_{t}C_s\| v\|_{H^{s}}.
\end{equation*}
\end{lema}
\begin{proof}
From Lemma \ref{descato}, it follows that $[S,W_\alpha(-t)v]S^{-1} \in 
\mathcal B(L^2)$ y $$\|[S,W_\alpha(-t)v]S^{-1}\|_{
\mathcal{B}(L^{2})} \leq C_s 
\|v\|_{H^{s}}.$$ Therefore $B(t,v)\in 
\mathcal{B}(L^2)$ and satisfies \eqref{2.3}

Proceeding as before \eqref{2.4} can be shown.
\end{proof}

\begin{lema}
$H^{s}(\re^2) \subset D(A(t,v))$ and $A(t,v)$ is a bounded operator from  
$Y=H^{s}( \re^2)$ to $X=L^2( \re^2)$ with 
$$\|A(t,v)\|_{\mathcal{B}(Y,X)} \leq \|v\|_{H^{s}},$$
for all $v\in Y$.  Furthermore, the function $t\mapsto A(t,v)$ is
strongly continuous from $\re$ to $\mathcal B(Y,X)$, for all
$v\in H^{s}$. On the other hand, the function $v\mapsto A(t,v)$
satisfies the following Lipschitz condition
$$\|A(t,v)-A(t,v')\|_{\mathcal{B}(Y,X)} \leq \|v-v'\|_X,$$
where $\mu$ is as the lemma above.
\end{lema}
\begin{proof}
  Inasmuch as $\{W_\alpha(t)\}$ is a unitary group in $L^2$, from the
  definition of $A(t,v)$, it follows that
  $H^{s}\left(\re^2\right) \subset D(A(t,v))$. In fact,
$$\begin{aligned}
\|A(t,v)f\|_{L^2} &= \|W_\alpha(-t)v\partial_x W_\alpha(-t)f\|_{L^2} 
\\ &\leq  C_{s}\|v\|_{H^{s}}\|\partial_xf\|_{L^2}\\ &\leq 
\|v\|_{H^{s}}\|f\|_{H^{s}},
\end{aligned}
$$ for all $f\in H^{s}$.

Now, for all $t,t'\in \re$ and $f,v \in H^{s}$, we have
\begin{align*}
\|A(t,v)f-A(t',v)f\|_{L^2} &\leq \|\left(W_\alpha(t)- W_\alpha(t')\right) 
W_\alpha(-t)v\partial_xW_\alpha(-t)f\|_{L^2}+\\
&+ \|(W_\alpha(-t)-W_\alpha(-t'))v\partial_x 
W_\alpha(-t)f\|_{L^2}+\\
&+\|W_\alpha(-t')v\partial_x (W_\alpha(-t)-W_\alpha(-t'))f\|_{L^2}.
\end{align*}
Since the group $\{W_\alpha(t)\}_{t\in\re}$ is strongly continuous,
$t\mapsto A(t,v)$ is strongly continuous from $\re$ to
$\mathcal{B}(H^{s},L^2)$. Finally, for any $t\in\re$, we have 
\begin{align*}
\|A(t,v)f-A(t,v')f\|_{L^2} &\leq \mu\|W_\alpha(-t)v 
-W_\alpha(-t)v'\|_{L^2}\|\partial_x W_\alpha(-t)f\|_{L^{\infty}}\\
&\leq \mu\|v-v'\|\, \|f\|_{H^{s}},
\end{align*}
this ends the lemma proof.
\end{proof}
If we take the open ball $W$ of $v\in H^{s}$ such that
$\|v\|_{H^s(\re^2)}<R$, the preceding lemmas show that the Cauchy
problem \eqref{eq:principal} satisfies the conditions of Theorem
\ref{tkato}. Therefore, for each $\psi \in H^{s}( \re^2)$, with $s>2$,
there exists $T>0$, that depends on $\|\psi\|_{H^{s}}$, and a unique
$v\in C([0,T], H^{s}(\re^2) ) \cap C^1([0,T], H^{s-1}(\re^2) )$
solution to the problem \eqref{problemaeq}. Moreover, the map
$\psi \to v$ is continuous from $H^{s}(\re^2)$ to
$C([0,T], H^{s}(\re^2) ) $. Now, from the group $W_\alpha(t)$ 
properties it can be verified that $u(t)=W_\alpha(-t)v$ is solution
to Cauchy problem \eqref{eq:principal} and satisfies all 
properties stated in the theorem.
\end{proof}
\begin{teo}
  The existence time of the solution of the Cauchy problem
  \eqref{eq:principal} can be chosen independently of $s$ in the
  following sense: if $u\in C([0,T], H^{s}( \re^2) )$ is the
  solution to \eqref{eq:principal} with $\psi\in H^{r}(\re^2) $, for
  some $r>s$,  then
  $u\in C([0,T], H^{r}(\re^2) )$.\par In particular, if
  $\psi \in H^{\infty}(\re^2)$, $u\in C([0,T], H^{\infty}(\re^2) )$
\end{teo}
\begin{proof}
  Let $r>s$, $u\in C([0,T],H^{r}(\re^2))$ solution to
  \eqref{eq:principal} and $v=W_\alpha(-t)u$. Suppose that
  $r\leq s+1$. If we apply $\partial_x^2$ on both sides of the
  differential equation \eqref{problemaeq}, we arrive to the following
  linear evolution equation for $w(t)=\partial_x^2v(t)$
\begin{equation}\label{ec3.1}
 \frac{dw}{dt}+A(t)w+B(t)w=0.
\end{equation}
where
\begin{align}
  A(t)&= \partial_x W_\alpha(t)u(t)W_\alpha(-t)\label{ec3.2}\\
  \intertext{and}
B(t)&= 2W_\alpha(t) u_x(t) W_\alpha(-t)\label{ec3.3}.
\end{align}

Since $v\in C\left([0,T];H^s\left( \re^2\right) \right)$, then
$w\in C\left([0,T];H^{s-2}\left( \re^2\right) \right)$. Besides
$w(0) = \psi_{xx}\in H^{r -2}\left(\re^2\right)$, because
$\psi \in H^r\left(\re^2\right)$. It is needed to see that
$w\in C\left([0,T];H^{r-2}\left( \re^2\right) \right)$. For this we
shall prove that the Cauchy problem associated to the linear equation
(\ref{ec3.1}) is locally well-posed for $1-s\leq k \leq s-1$, for
which we have the following lemma whose proof is similar to that of
Lemma 3.1 in \cite{katoMM1979}.
\begin{lema}\label{lema2.6}
  The family $\{A(t)\}_{0\leq t\leq T}$ has a unique family of
  evolution operators associated,
  $\{U(t,\tau)\}_{0\leq t\leq \tau \leq T}$, in the spaces $X= H^h$,
  $Y=H^k$, where
\begin{equation}
-s\leq h\leq s-2 \quad 1-s\leq k \leq s-1 \quad k+1 \leq h.
\end{equation}
In particular, $U(t,\tau): H^r \to H^r$ for $-s\leq s\leq s-1.$
\end{lema}
Then, $w$ satisfies the equation
\begin{equation}\label{ec2.10}
w(t)=U(t,0)\psi_{xx}+\displaystyle \int_0^t 
U(t,\tau)[-B(\tau)w(\tau)+f(\tau)]\, 
d\tau.
\end{equation}
Since $\psi_{xx}\in H^{r-2},$ $B(t)$, given by \eqref{ec3.3}, is an 
operators family in $ H^{r-2}$ that is strongly continuous for $t$ in
the interval $[0,T]$. From Lemma \ref{lema2.6}, the solution to
(\ref{ec2.10}) belongs to
$C\left([0,T];H^{r-2}\left( \re^2\right) \right)$. In other words,
$\partial_{x}^2 u\in C\left([0,T];H^{r-2}\left( \re^2\right) \right)$.

If $w_1(t)=\partial_{x}\partial_{y}v(t)$, we have 
\begin{equation}
\frac{dw_1}{dt}+A(t)w_1+B_1(t)w_1=f_1(t),
\end{equation}
where
\begin{align}
  B_1(t)&=\mathcal{W}(t)u_x(t)\mathcal{W}(-t)=\frac12 B(t),\\
  \intertext{and}
f_1(t)&=-\mathcal{W}(t)\left(u_{xx}(t)u_y(t)\right).
\end{align}
As before, we have
\begin{equation}
w_1(t)=U(t,0)\psi_{xy}+\displaystyle \int_0^{t} U(t,\tau)(-B_1(\tau)w_1( 
\tau )+f_1(\tau))\, d\tau.
\end{equation}
Since $u_{xx}\in C\left([0,T];H^{r-2}\left( \re^2\right) \right)$, $f_1 \in 
C\left([0,T];H^{r-2}\left( \re^2\right) \right)$. Inasmuch as, also, $B_1(t)\in 
\mathbf(H^{r-2}\left( \re^2\right))$ is strongly continuous on the interval 
$[0, T]$. Arguing as before we have that $w_1 \in 
C\left([0,T];H^{r-2}\left( \re^2\right) \right)$, or equivalently 
$u_{xy} \in C\left([0,T];H^{r-2}\left( \re^2\right) \right)$.\par
Analogously, if $w_2(t)=\partial_y^2v(t)$, we have that 
\begin{equation}
\frac{dw_2}{dt}+A(t)w_2=f_2(t),
\end{equation}
where
\begin{equation}
f_2(t)=-2\mathcal{W}(t)( 
u_{xy}u_y(t).
\end{equation}
Therefore,
\begin{align}
w_2(t)=U(t,0)\psi_{yy}+\int_0^t U(t,\tau)f_2(\tau)\, d\tau. 
\end{align}
Since $u_{xy} \in C\left([0,T];H^{r-2}\left( \re^2\right) \right)$, $f_2 \in 
C\left([0,T];H^{r-2}\left( \re^2\right) \right)$. Repeating the argument
above, we can conclude that $w_2 \in C\left([0,T];H^{r-2}\left( \re^2\right) 
\right)$, or equivalently, $\partial_y^2u \in C\left([0,T];H^{r-2}\left( 
\re^2\right) \right)$.

Hence, we have showed that if $s< r \leq s+1$ and $\psi \in H^r$,  $u 
\in C\left([0,T];H^{r}\left( \re^2\right) \right)$. To see the case 
$r> s+1$, since $\psi \in H^{s'}$, for $s' < r$, we can use over and
over again what we have proved so far, to obtain that $u\in 
C\left([0,T];H^{r}\left( \re^2\right) 
\right)$
\end{proof}
\subsection{Local well-posedness in $X^{s}$, $\widehat X_\alpha^{s}$ and $Y^s$}
Let us finish this section showing the local well-posedness of
\eqref{eq:principal} in the spaces $X^{s}$, $\widehat X_\alpha^{s}$ and $Y^s$.
\begin{teo}\label{teo:2.7}
  Let $s$ and $\alpha$ be as in Theorem \ref{principal}. Let 
  $Z$ also be any of the spaces $X^{s}$, $\widehat X^{s}$, $\widehat
  X_\alpha^{s}$,  $Y^s$ and $\widehat Y^{s}$. Then, if $\psi\in Z$ and
  $u\in C([0,T], H^{s})$ is 
  solution to \eqref{eq:principal} with $u(0)=\psi$, 
  $u\in C([0,T], Z)$. Moreover, $\psi \mapsto u$ is continuous from
  $Z$ to $C([0,T], Z)$
\end{teo}
\begin{proof}
  Suppose that $\psi$ and $u$ are as in the hypothesis of the
  theorem. Then, from the fundamental calculus theorem, we have that
  $$ u(t)=W_\alpha(-t)\psi +\int_0^t W_\alpha(-t+\tau
  )\partial_x\left( \frac 
  {u^{2} (\tau)} {2}\right) \,
  d\tau .$$ So that, $$ \partial_x^{-1}
  u(t)=W_\alpha(-t)\partial_x^{-1}\psi +\int_0^t 
  W_\alpha(-t+\tau ) \left( \frac {u^{2} (\tau)} {2}\right) \,
  d\tau .$$ From here it follows that \eqref{eq:principal} is locally
  well-posed in the spaces $\widehat X^s$ and $X^s$.\par On
  the other hand,
  $$ \partial_x^{-1} \partial_y
  u(t)=W_\alpha(-t)\partial_x^{-1}\partial_y \psi +\int_0^t
  W_\alpha(-t+\tau ) ( u\partial_y u) (\tau) \, d\tau .$$ Then,
  \eqref{eq:principal} is locally well-posed in the space
  $\widehat Y^s$.\par  To see the local well-posedness in $Y^s$ is
  slightly more complicated. Let $v$ be as in
  \eqref{problemaeq}. Hence, $w=\partial_x^{-1} \partial_y v$
  satisfies
  \begin{equation}\label{problemaeqn}
    \begin{cases}
      w_t+A(t,v)w=0\\
      w(0)=\partial_x^{-1} \partial_y\psi,
    \end{cases} 
  \end{equation}
  where $A(t,v)$ is as in \eqref{problemaeq}. From linear case of
  Kato's theory, more specifically from Theorem \ref{tkato}, when
  taking $X=L^2$ and $Y=H^s$, we have that, if
  $\partial_x^{-1} \partial_y \psi\in H^s$, $w$ is continuous on $t$
  and depends continuously on this data. This shows the local
  well-posedness of the problem \eqref{eq:principal} in $Y^s$.\par
  Analogously we can show that \eqref{eq:principal} is local
  well-posedness in space $X_\alpha^s$.
\end{proof}

%%% Local Variables:
%%% mode: latex
%%% TeX-master: "allart"
%%% End:

%% file: resumen2art.tex
\section{Local well-posedness in low regularity spaces}
In this section we examine the local well-posedness of equation
\eqref{eq:principal} in $H^{s_1,s_2}(\re^2)$, for
$-1\le \alpha \le 1$, $s_1>\frac{17}{12}-\frac{\alpha}{4}$ and 
$1<s_2\leq s_1$. We will use the dispersive properties of the group
generated by the linear equation associated to
\eqref{eq:principal}. This is the same strategy used by Kenig in
\cite{kenig2004} for the KP-I equation and by Linares, Pilod and Saut
in \cite{linpilsaut3} for the f-KPI and f-KPII equations.
\subsection{Linear estimates}
The linear Cauchy problem associated to the problem \eqref{eq:principal} is  
\begin{equation}\label{eq:principallineal}
\begin{cases} 
 u_{t}=u_{xxx}-\mathscr{H}D_{x}^{\alpha}u_{yy}, \cr
 u(0)=\psi. 
\end{cases}
\end{equation} 
where $-1\le \alpha \le 1$. The unitary group generated by this
problem is 
\begin{equation}\label{grupo}
	W_\alpha(t)\psi(x,y)=\left(e^{-it(\xi^3+\sgn(\xi)\left| \xi
            \right|^\alpha \eta^2)}\widehat{\psi}\right)^\vee
        = S_t^\alpha*\psi (x,y),
\end{equation}
where
\begin{equation}
	S_t^\alpha(x,y)=\int_{\re^2 } {e^{-it ( \xi^3 +\sgn( \xi)
            \left| \xi\right|^\alpha\eta^2)+ix\xi+iy\eta}}d\xi 
        d\eta 
\end{equation}
We shall examine this group properties.
\begin{lema}\label{intosc}
  Let $-1\leq \alpha\leq 1$ and  $\frac{\alpha }{2}- 1< \Real \beta< \frac{
    \alpha }{2}$. Then,
\begin{equation}
	\left\|D_x^{\beta}W_\alpha (t) \psi\right\|_{L^ \infty(\re^2)}
        \lesssim \left| t\right|^{-\frac{5+2 \beta-
            \alpha}{6}}\left\| \psi\right\|_{L^1(\re^2)}.
\end{equation}
\end{lema}
\begin{proof}
Effectively,
\begin{equation}\label{eq:2.5a}
\begin{split}
	D_x^\beta S_t^\alpha(x,y)&=\int_{\re^2}{\left|\xi\right|^\beta
          e^{-it(\xi^3+\sgn(\xi)\left|\xi\right|^\alpha\eta^2)+ix\xi+iy\eta}}d\xi
        d\eta \\
   &=\int_{\re_\xi}{\left|\xi\right|^\beta e^{-it\xi^3+ix\xi
     }\int_{\re_\eta }{e^{-it\sgn(\xi)\left|\xi \right|^\alpha
         \eta^2+iy \eta}}} d\eta d\xi \\
	&=\int_{\re_\xi}{ \left|\xi\right|^\beta e^{-it\xi^3
            +ix\xi+\frac{i \sgn(\xi)y^2}{4t\left|\xi\right|^\alpha}}
          \int_{\re_\eta}{e^{
              -i\left|t\right|\sgn(t\xi)\left|\xi\right
              |^\alpha\left[\eta-\frac{ y}{2t\sgn( \xi) \left| \xi
                  \right |^\alpha}\right]^2}}}
        d\eta d\xi\\
	&=\frac{\pi^\frac{1}{2}}{\left|t\right |^\frac{1}{2}}
        \int_{\re}{ e^{-it\xi^3+ix\xi+ \frac{i\sgn( \xi)y^2}{4t\left|
                \xi \right |^\alpha }-i \sgn(t\xi)\frac{ \pi}{4}}}
        \left | \xi\right|^{\beta-\frac{\alpha}{2}} d\xi\\
        & =\frac{\pi^\frac{1}{2}}{|t |^\frac{5+2\beta- \alpha}{6}}\int_\re
        {e^{i\left( \theta^3-x\theta t^{-\frac{1}{3}} -\frac
            {\sgn(\theta)y^2}{4| \theta
              |^\alpha}|t |^{ \frac{\beta}{3}-1} \right)}|\theta
        |^{\beta-\frac{\alpha}{2}} }d\theta .
\end{split}
\end{equation}
Let us see that the last integral is bounded. For this let us take a
function $\chi$ defined on all $\re$, infinitely differentiable, with
support in the interval $[-2,2]$ such that $\chi\equiv 1$ in the
interval $[-1,1]$. So,
\begin{multline}\label{eq:2.5}
     \int_\re {e^{i\left( \theta^3-x\theta t^{-\frac{1}{3}} -\frac
            {\sgn(\theta)y^2}{4| \theta
              |^\alpha}|t |^{ \frac{\beta}{3}-1} \right)}|\theta
        |^{\beta-\frac{\alpha}{2}} }d\theta =\\ =\int_\re{e^{i\left(
            \theta^3-x\theta t^{-\frac{1}{3}} -\frac 
            {\sgn(\theta)y^2}{4| \theta
              |^\alpha}|t |^{ \frac{\beta}{3}-1} \right)}|\theta
        |^{\beta-\frac{\alpha}{2}} } \chi(\theta )d\theta+ \\
	+\int_\re{e^{i\left( \theta^3-x\theta t^{-\frac{1}{3}} -\frac
            {\sgn(\theta)y^2}{4| \theta
              |^\alpha}|t |^{ \frac{\beta}{3}-1} \right)}|\theta
        |^{\beta-\frac{\alpha}{2}} }(1- \chi(\theta ))d\theta.
\end{multline}
Clearly the first integral on the right hand side is bounded. To see
that the second integral on the right hand side is bounded, we will
make use the Van der Corput lemma (see \cite{linaresponce}, Corollary
1.1). The second and third derivatives of phase function
$\vartheta_\alpha (\theta) =\theta^3-x\theta t^{-\frac{1}{3}} -\frac
{\sgn(\theta)y^2}{4| \theta |^\alpha}|t |^{ \frac{\beta}{3}-1} $ in
the integral are
\begin{align*}
\vartheta_\alpha ''(\theta)=6\theta+ \sgn( \theta) \alpha( \alpha+ 1)
\frac{ y^2\left|t\right |^{\frac{\beta}{3}-1}}{4}\left |\theta \right
|^{ -\alpha-2} \\\intertext{and} \vartheta_\alpha '''(\theta)=6 +\alpha( \alpha+ 1)(
\alpha+ 2)\frac{y^2 \left|t\right|^{\frac{ \beta}{3}-1}}{4} \left|
  \theta \right|^{-\alpha-3}.
\end{align*}
It is easily verified that, for $|\theta|\ge 1$,
$|\vartheta_\alpha'' (\theta)|\geq 6$ when $-1\leq\alpha\leq 0$ and
that $ |\vartheta_\alpha'''(\theta)|\geq 6$ when $0\leq\alpha\leq
1$. Since the function
$\theta \mapsto |\theta|^{\beta-\frac{\alpha}{2}}(1-\chi(\theta))$ is
uniformly bounded and integrable on the set $|\theta|\ge 1$, it is
verified that the second integral on the right side of \eqref{eq:2.5}
is bounded. Which verifies that the last integral in
\eqref{eq:2.5a} is uniformly bounded. Therefore,
\begin{equation*}
	\|D_x^\beta S_t^\alpha(x,y)\|_{L^\infty(\re^2) } \lesssim
        | t|^{\frac{5+2\beta-\alpha}{6}}
\end{equation*}
The theorem follows immediately from Young's inequality for
convolution.
\end{proof}
\begin{obs}
  The last result coincides with that proved by Linares and Pastor in
  \cite{linares4} for the ZK equation (case $\alpha=1$). The
  results of Saut in \cite{saut1993remarks}, for the KPI equation 
  (case $\alpha=-1$), and of Lizarazo in his PhD thesis (see \cite{tesjul})
  (case $\alpha=0$) are better estimates.
\end{obs}
\begin{corol}\label{corointosc}
Let $-1\leq\alpha\leq 1$, $0<\epsilon<1$ and $0\leq\theta\leq 1$. Then,
\begin{equation}
\left\|D_x^{\theta\left( \frac{\alpha}{2} -\epsilon \right) }W_\alpha (t) \psi
\right\|_{L^p(\re^2 )}\leq
\left|t\right|^{-\frac{\theta (5 -2\epsilon)} {6}}\left\|\psi\right \|_{L^{p'}(\re^2 )}
\end{equation}
where $\frac{1}{p}+\frac{1}{p'}=1$ and $p=\frac{2}{1-\theta}$.
\begin{proof}
  Let $\psi\in L^1(\re^2)\cap L^2(\re^2)$. The corollary follows from
  Stein's interpolation theorem for analytical operator families (see
  \cite{steinweiss1971}). Let us set
  $z=\theta +i \gamma\in \mathbb{C}$. For each $z$ we define the
  operators $T^z$ by
$$T_z\psi=D_x^{z(\frac{\alpha}{2}-\epsilon)}W_\alpha(t)\psi.$$ 
This family $\left\{T_z\right\}$ is an admissible family of
operators. By Lemma \ref{intosc}, we have  
\begin{equation}
  \left\|T_{1+i\gamma}\psi\right\|_{L^\infty(\re^2
    )}=\left\|D_x^{(\frac{\alpha}{2}-\epsilon)(1+i\gamma)}W_\alpha(t)\psi
  \right\|_{L^\infty(\re^2
    )}\leq c  \left|t\right|^{-(\frac{5-2\epsilon}{6})}\left\|\psi
  \right \|_{L^1(\re^2 )}.
\end{equation}
Also, since $\left\{W_\alpha(t)\right\}_{t\in\re}$ is an strongly continuous 
group on the parameter $t$ in $L^2(\re^2)$, we have
\begin{equation}
	\left\|T_{i\gamma}\psi \right \|_{L^2(\re^2 )}=\left \|D_x^{
            (\frac{ \alpha}{2}-\epsilon)( i\gamma) }W_\alpha(t) \psi
        \right \|_{L^2(\re^2 )}=  \left\|\psi\right\|_{L^2(\re^2 )}.
\end{equation}
From Stein's interpolation theorem we obtain 
\begin{equation}
	\left\|T_{\theta}\psi\right\|_{L^p(\re^2 )}\leq c \left|t\right|^{-\theta(\frac{5-2\epsilon}{6})}\left\|\psi\right\|_{L^{p'}(\re^2 )},
\end{equation}
what was we wanted to prove.
\end{proof}
\end{corol}

\begin{corol}\label{corostrichartznolineal}
Let $-1\leq\alpha\leq 1$, $0<\epsilon<1$ and $0<\theta\leq 1$. Then,
\begin{equation}\label{eqstrichartznolineal}
	\left\|D_x^{\frac{12}{q\left(5-2\epsilon\right)}\left(\frac{\alpha}{2}-
              \epsilon\right)}\int_{-\infty}^{\infty}{W_\alpha (t-t') F(t')} dt'
        \right \|_{L_T^qL_{xy}^p}\lesssim \left\|F\right\|_{L_T^{q'}L^{p'}_{xy}},
\end{equation}
where $\frac{1}{q}+ \frac{1}{q'}=\frac{1}{p}+ \frac{1}{p'}=1$,
$p=\frac{2} {1-\theta}$, $\frac{2}{q}= \frac{ \theta(5-2\epsilon)}{6}$
and $\frac{1}{p }+\frac{1}{q}=\frac{1}{2}-\frac{\theta(1+2\epsilon)}{12}$.
\begin{proof}
From Minkowski's inequality, Corollary \ref{corointosc} and from the
Hardy-Littlewood-Sobolev theorem, we have
\begin{multline*}
\left\|D_x^{\theta\left(\frac{\alpha}{2}-\epsilon\right)}\int_{-\infty}^\infty{W_\alpha
    (t-t') F(t')}dt'\right\|_{
  L_t^qL_{xy}^p}\\=\left\|\left\|\int_{-\infty}^\infty D_x
    ^{\theta\left(\frac{\alpha}{2}-
        \epsilon\right)}W_\alpha(t-t')F(\cdot, \cdot,t') dt'
  \right\|_{L_{xy}^p} \right\|_{L_t^q}\\
\leq \left\|\int_{-\infty}^\infty \left\|D_x^{\theta
      \left(\frac{\alpha}{2}-\epsilon \right)}W_\alpha(t-t')F( \cdot,
    \cdot,t') \right\|_{L_{xy}^p} dt'\right\|_{L_t^q}\\
\lesssim \left\|\int_{-\infty}^\infty{\left| t-t'\right|^{-\frac{
        \theta(5-2\epsilon )}{6}}\left\|F( \cdot, \cdot,t') \right
    \|_{L_{ xy}^{p'}}}dt'\right\|_{L_{t^q}}\\
\lesssim \left\|\left\|F(\cdot,\cdot,t')\right\|_{L_{xy}^{p'}}
\right\|_{L_{ t^{q'}}}=\left\|F\right\|_{L_t^{q'}L_{xy}^{p'}}.
\end{multline*}
\end{proof}
\end{corol}
Using the Stein-Thomas argument  (see \cite{ginibre1992}) we have.
\begin{prop}\label{propstrichartzlineal}
  Let $\alpha$, $\epsilon$, $\theta$, $p$ and $q$ be as in the
  previous corollary. Then,
\begin{equation}\label{eqstrichartzlineal}
\left\|D_x^{\frac{6}{q\left(5-2\epsilon\right)}\left(\frac{\alpha}{2}-\epsilon\right)}W_\alpha(t)\psi\right\|_{L_T^qL_{xy}^p}\lesssim \left\|\psi\right\|_{L_{xy}^2}.
\end{equation}
\end{prop}
From the above we have the following two very useful corollaries in
the proof of the local well-posedness in this section.
\begin{corol}\label{coronormalt2lxyinf}
For each $T>0$ and $0<\epsilon<1$, we have that
\begin{equation}\label{eqnormalt2lxyinf}
	\left\|W_\alpha(t)\psi\right\|_{L_T^2L_{xy}^\infty}\lesssim T^{\frac{1+2\epsilon}{12}}\left\|D_x^{\frac{1}{2}(\epsilon-\frac{\alpha}{2})}\psi\right\|_{L_{xy}^2} .
\end{equation}
\begin{proof}
From  H\"older's inequality and Proposition
\ref{propstrichartzlineal} we have that
\begin{equation*}
\begin{split}
\left\|W_\alpha(t)\psi\right\|_{L_T^2L_{xy }^ \infty}&\leq T^{ \frac{
    1+ 2\epsilon}{12}}\left \|W_\alpha( t) \psi\right\|_{L_T^{
    \frac{12}{ 5-2 \epsilon}}L_{xy}^\infty}\lesssim T^{ \frac{1+2
    \epsilon}{ 12}}\left\|D_x^{\frac{1}{2}( \epsilon-
    \frac{\alpha}{2})} \psi\right\|_{L_{xy}^2}.  
\end{split}
\end{equation*}
\end{proof}
\end{corol}
Just like in \cite{kenig2004} and \cite{linpilsaut3}, we prove a
refined Strichartz estimate for the solution to non-homogeneous linear problem
\begin{equation}\label{eqlinealf}
\partial_t w=w_{xxx}-\mathscr{H}D_x^{\alpha}w_{yy}+F.
\end{equation}
So, we have the next lemma.
\begin{lema}\label{lemastrichartzref}
  Let $-1\leq\alpha\leq 1, 0<\epsilon<1$ and $T>0$. If $w$ is solution
  to (\ref{eqlinealf}), then there exists $c_\epsilon>0$ such that
\begin{equation}\label{eqstrichartzref}
\left\|\partial_x w\right \|_{ L_T^1 L_{xy}^\infty} \leq c_\epsilon
T^{ \frac{7+2\epsilon}{12}} \left(\sup_{t\in[0,T]}{ \left\|J
      _x^{\frac{ 17}{12}-\frac{\alpha}{ 4}+ \frac{ 2\epsilon}{ 3}}
      w\right \|_{L_{xy}^2}}+ \int_0^T \left\|D_x^{ \frac{5}{12}-
      \frac{ \alpha}{4}+\frac{2\epsilon}{3}} F(t) \right\|_{L_{xy}^2}dt\right).
\end{equation}
\begin{proof}
  We will make use of a Littlewood-Paley type decomposition of $w$ on $x$
  variable. For this, let $\varphi_0,\varphi\in C^\infty_0$ with
  $\supp(\varphi_0)= \left\{\left|\xi\right|<2\right\}$ and 
  $\supp(\varphi)= \left\{\frac{1}{2}<\left|\xi\right|<2\right\}$ such
  that $\varphi_0(\xi)+\sum_{k=1}^\infty{\varphi(2^{-k}\xi)}=1$ for
  all $\xi \in \re$.  Let us define $P_k w$ via Fourier transform
  by $\widehat{P_{0}w}(\xi,\eta)=\varphi_0(\xi)\widehat{w}(\xi,\eta)$
  and
  $\widehat{P_{k}w}(\xi,\eta)=\varphi(2^{-k}\xi)\widehat{w}(\xi,\eta)$
  for $k\geq 1$, in such a way that
\begin{equation*}
w=\sum_{k=0}^\infty{P_k w}.
\end{equation*}
Let us first make an estimate for
$\left\|\partial_x P_0 w \right \|_{ L_T^1L_{xy}^\infty}$. Since $w$
satisfies \eqref{eqlinealf}, then $P_0 w$ satisfies the integral
equation when apply $P_0$ on both sides of the equation. So, from
the Cauchy-Schwarz inequality and \eqref{eqnormalt2lxyinf}, it
follows that
\begin{equation}\label{estp0}
\begin{split}
&\left\|\partial_x P_0 w\right\|_{L_T^1L_{xy}^\infty}\leq\\&
\leq \left\|W_\alpha(t)\partial_x P_0 w(0)\right\|_{L_T^1L_{xy}^\infty} +\int_0^t\left\|W_\alpha(t-t')\partial_xP_0F(t')\right\|_{L_T^1L_{xy}^\infty} dt'\\
&\leq T^\frac{1}{2}\left(\left\|W_\alpha(t)\partial_x P_0 w(0)\right\|_{L_T^2L_{xy}^\infty} +\int_0^t\left\|W_\alpha(t-t')\partial_xP_0F(t')\right\|_{L_T^2L_{xy}^\infty} dt'\right)\\
&\lesssim T^{\frac{1}{2}+\frac{1+2\epsilon}{12}}\left(\left\|D_x^{\frac{1}{2}(\epsilon-\frac{\alpha}{2})}\partial_x P_0 w(0)\right\|_{L_{xy}^2} +\int_0^t\left\|D_x^{\frac{1}{2}(\epsilon-\frac{\alpha}{2})}\partial_xP_0F(t')\right\|_{L_{xy}^2} dt'\right)\\
&\lesssim c_\epsilon T^\frac{7+2\epsilon}{12}\left(\left\|J_x^{\frac{17}{12}-\frac{\alpha}{4}+\frac{2\epsilon}{3}} P_0 w(0)\right\|_{L_{xy}^2} +\int_0^T\left\|D_x^{\frac{5}{12}-\frac{\alpha}{4}+\frac{2\epsilon}{3}}P_0F(t)\right\|_{L_{xy}^2} dt\right).
\end{split}
\end{equation}

Now, let us estimate
$\left \| \partial_x P_k w \right \|_{L_T^1L_{ xy}^\infty }$, when
$k\geq 1$. For this, we will make a suitable partition in time in such
a way that allow us to control the localized frequencies associated to
$x$ variable. So, let
$\mathcal{P}=\left\{a_0, a_1 ,... , a_{2^k}\right\}$ be a partition
from interval $[0,T]$ with $a_j=jT 2^{-k}$, $j=0, 1, \cdots,2^k$.  We
denote by $I_j$ the interval $[a_{j-1},a_j]$.\par
Then, thanks to Young's inequality, Cauchy-Schwarz's inequality and
the inequality
$\left\|(\chi_{\left\{2^{k-1}<\left|\xi\right|<2^k\right\}}\xi)^\vee
\right\|_{L^1(\re)}\lesssim 2^k,$ we have
\begin{equation}\label{pklt1lxyinf}
\begin{split}
\left\|\partial_xP_k w\right\|_{L_T^1L_{xy}^\infty}&\leq
\left\|(\chi_{\left\{2^{k-1}<\left|\xi\right|<2^k\right\}}(\xi)\left|\xi\right|)^\vee
\right\|_{L^1(\re)}\left\|P_k w\right\|_{L_T^1L_{xy}^\infty}\\ 
&\lesssim 2^k\sum_{j=1}^{2^k}\left\|P_k w\right\|_{L_{I_j}^1L_{xy}^\infty}\\
&\lesssim (2^k)\sum_{j=1}^{2^k}(a_j-a_{j-1})^\frac{1}{2}\left\|P_k
  w\right\|_{L_{I_j}^2L_{xy}^\infty}\\ 
&=(2^kT)^\frac{1}{2}\sum_{j=1}^{2^k}\left\|P_k w\right\|_{L_{I_j}^2L_{xy}^\infty}.
\end{split}
\end{equation}
From Duhamel's principle on each interval $[a_{j-1},a_j]$, we have that,
for each $t\in[a_{j-1},a_j]$, 
\begin{equation*}
P_kw(t)= W_\alpha(t-a_j)P_kw(\cdot,a_j)+\int_{a_{j-1}}^tW_\alpha(t-t')P_kF(t')dt'.
\end{equation*}
By \eqref{pklt1lxyinf} and \eqref{eqnormalt2lxyinf}, we get 
 \begin{equation}\label{estp1}
\begin{split}
&\left\|\partial_xP_k w\right\|_{L_T^1L_{xy}^\infty}\lesssim (2^kT)^\frac{1}{2}\sum_{j=1}^{2^k}\left\|P_k w\right\|_{L_{[a_j,b_j]}^2L_{xy}^\infty}\\
&\lesssim (2^kT)^\frac{1}{2}\sum_{j=1}^{2^k}{\left(\left\|W_\alpha(t-a_j)P_k w(\cdot,a_j)\right\|_{L_{I_j}^2L_{xy}^\infty}+\int_{a_{j-1}}^t{\left\|W_\alpha (t-t')P_k F(t')\right\|_{L_{I_j}^2L_{xy}^\infty}}dt'\right)}\\
&\lesssim 2^\frac{k(5+2\epsilon)}{12}T^\frac{7+2\epsilon}{12}\sum_{j=1}^{2^k}{\left(\left\|D_x^{\frac{1}{2}\left(\epsilon-\frac{\alpha}{2}\right)}P_k w(\cdot,a_j)\right\|_{L_{xy}^2}+\int_{I_j}{\left\|D_x^{\frac{1}{2}\left(\epsilon-\frac{\alpha}{2}\right)}P_k F(t')\right\|_{L_{xy}^2}}dt'\right)}\\
&\lesssim
T^\frac{7+2\epsilon}{12}\Bigg\{2^\frac{k(5+2\epsilon)}{12}\sum_{j=1}^{2^k}{\left\|D_x^{\frac{1}{2}\left(\epsilon-\frac{\alpha}{2}\right)}P_k
      w(\cdot,a_j)\right\|_{L_{xy}^2}}+\\
&\phantom{T^\frac{7+2\epsilon}{12}\Bigg\{2^\frac{k(5+2\epsilon)}{12}\sum_{j=1}^{2^k}\left\|D_x^{\frac{1}{2}\left(\epsilon-\frac{\alpha}{2}\right)}P_k
      w\right.} + 2^\frac{k(5+2\epsilon)}{12}\int_{0}^T{\left\|D_x^{\frac{1}{2}\left(\epsilon-\frac{\alpha}{2}\right)}P_k F(t')\right\|_{L_{xy}^2}}dt'\Bigg\}\\
&\lesssim c_\epsilon T^\frac{7+2\epsilon}{12}\left(\sup_{t\in[0,T]}{\left\|D_x^{\frac{17}{12}-\frac{\alpha}{4}+\frac{2\epsilon}{3}}P_k w(t)\right\|_{L_{xy}^2}}+\int_0^T{\left\|D_x^{\frac{5}{12}-\frac{\alpha}{4}+\frac{2\epsilon}{3}}P_k F(t')\right\|_{L_{xy}^2}}dt'\right).
\end{split}
\end{equation}
For the sake of the completeness, let us explain the
inequality
\begin{equation}\label{eq:np1}
2^\frac{k(5+2\epsilon)}{12} \sum_{j=1}^{2^k}{ \left\|D_x^{
      \frac{1}{2}\left(\epsilon- \frac{ 
          \alpha}{2} \right)}P_k w(\cdot,a_j)\right\|_{L_{xy}^2}} \le
  \sup_{t\in[0,T]}{\left\|D_x^{ \frac{17}{12} -\frac{ \alpha}{4}+
        \frac{2 \epsilon}{3}}P_k
      w(t)\right\|_{L_{xy}^2}}, 
\end{equation}
what was used to prove the inequality above. Observe that,
for each integer $j\le 2^k$, we have 
$$\left\|D_x^{ \frac{1}{2}\left(\epsilon- \frac{ \alpha}{2}
    \right)}P_k w(\cdot,a_j)\right\|_{L_{xy}^2}\le
j^{-\frac{k(17+2\epsilon)}{12}} \left\|D_x^{\frac{17}{12} -\frac{
      \alpha}{4}+ \frac{2 \epsilon}{3}} P_k
  w(\cdot,a_j)\right\|_{L_{xy}^2}. $$ Summing on $j$ from $1$ to
$2^k$ and considering 
$$ \sum_{j=1}^{2^k} j^{-\frac{k(17+2\epsilon)}{12}}\sim
2^{-\frac{k(5+2\epsilon)}{12}} $$ it follows the inequality
\eqref{eq:np1}.\par \eqref{estp1} together with \eqref{estp0} show the
present lemma.
\end{proof}
\end{lema}
\subsection{Energy Estimate}
\begin{lema}\label{lemaestenergia}
  Let $-1\leq \alpha\leq 1$, $T>0$ and assume that  $u\in C([0,T ]; H^ \infty(
  \re^2) )$ is a solution to the Cauchy problem
  \eqref{eq:principal}. Then, there exists a positive constant $C$
  such that, for $1\leq s_2\leq s_1$, 
\begin{equation}\label{eqestenergia}
\left\|u\right\|_{L_T^\infty H_{xy}^{s_1,s_2}}\leq \left\|\psi\right\|_{H^{s_1,s_2}_{xy}}e^{{C(\left\| \partial_x u\right\|_{L_T^1L_{xy}^\infty}+\left\| \partial_y u\right\|_{L_T^1L_{xy}^\infty})}}.
\end{equation} 
\begin{proof}
  Let $u$ as in the statement in lemma. Let us estimate first
  $\left\|J_x^{s_1} u\right\|_{L^2_{xy}}$. Operating with $J_x^{s_1}$ and
  then multiplying by $J_x^{s_1} u$ on both sides of the equation
  \eqref{eq:principal}, and integrating with respect to $x,y$, we obtain
\begin{equation}\label{eqestenergia6}
  \begin{split}
\frac{1}{2}\frac{d}{dt}\int_{\re^2}{{\left(J_x^{s_1}u\right)}^2  dx dy
}&=\int_{\re^2}{J_x^{s_1}(u\partial_x u) J_x^{s_1} u\, dx dy }\\ 
	&=\int_{\re^2} \left[J_x^{s_1},u\right]\partial_x u J_x^{s_1}
        udxdy+\int_{\re^2} u J_x^{s_1}\partial_x u J_x^{s_1} u\,dxdy.
      \end{split}
\end{equation}
From inequalities of Cauchy-Schwarz and
Kato-Ponce (Lemma \ref{katoponce}), applied in the $x$  variable, we have
\begin{equation}\label{eqestenergia7}
\begin{split}
	\int_{\re^2} \left[J_x^{s_1},u\right]\partial_x u J_x^{s_1}
        u\,dxdy&\leq \int_{\re_y} \left\|\left[J_x^{s_1},u\right]\partial_x
          u\right\|_{L_{x}^2}\left\| J_x^{s_1} u\right\|_{L_{x}^2}dy\\
        &\lesssim \int_{\re_y} \left\| \partial_x
          u\right\|_{L_{x}^\infty}  \left\| J_x^{s_1} u\right
        \|^2_{L_{ x}^2}dy \\ &\leq  \left\| \partial_x
          u\right\|_{L_{xy}^\infty}  \left\| J_x^{s_1}
          u\right\|^2_{L_{xy}^2}\leq  \left\| \partial_x
          u\right\|_{L_{xy}^\infty}   \left\|  u\right\|^2_{H_{xy}^{s_1,s_2}}.
\end{split}
\end{equation}
On the other hand, 
\begin{equation}\label{eqestenergia9}
 \int_{\re^2} u J_x^{s_1}\partial_x u J_x^{s_1}\, u\,dxdy=-\frac{1}{2}
 \int_{\re^2} \partial_xu \left( J_x^{s_1}u\right)^2 dxdy \lesssim \left\| \partial_x u\right\|_{L_{xy}^\infty}  \left\| J_x^{s_1} u\right\|^2_{L_{xy}^2}.
\end{equation}
By \eqref{eqestenergia6}, \eqref{eqestenergia7} and
\eqref{eqestenergia9}, we have
\begin{equation}\label{eqestenergia10}
	\frac{d}{dt}\left\|J_x^{s_1}u\right\|^2_{L^2_{xy}}\lesssim \left\| \partial_x u\right\|_{L_{xy}^\infty}  \left\| J_x^{s_1} u\right\|^2_{L_{xy}^2}\lesssim \left\| \partial_x u\right\|_{L_{xy}^\infty}  \left\|  u\right\|^2_{H_{xy}^{s_1,s_2}}.
\end{equation}
In a totally analogous way we shall estimate
$\left\| J_y^{s_2} u \right\|_{ L^2_{xy}}$. So,
\begin{equation}\label{eqestenergia11}
	\frac{1}{2}\frac{d}{dt}\left\|J_y^{s_2}u\right\|^2_{L^2_{xy}}=\int_{\re^2}
        \left[J_y^{s_2},u\right]\partial_x u J_y^{s_2} u\,
        dxdy+\int_{\re^2} u J_y^{s_2}\partial_x u J_y^{s_2} u\, dxdy.
\end{equation}
Proceeding as before, we have
\begin{equation}\label{eqestenergia12}
\begin{split}
	&\int_{\re^2} \left[J_y^{s_2},u\right]\partial_x u J_y^{s_2}
        u\,dxdy\leq \int_{\re_x}
        \left\|\left[J_y^{s_2},u\right]\partial_x
          u\right\|_{L_{y}^2}\left\| J_y^{s_2}
          u\right\|_{L_{y}^2}dx\lesssim  \\ &\lesssim \int_{\re_x} \left(\left\| \partial_y u\right\|_{L_{y}^\infty}\left\|J_y^{s_2-1}\partial_x u\right\|_{L^2_y} +\left\| \partial_x u\right\|_{L_{y}^\infty}\left\|J_y^{s_2}\partial_x u\right\|_{L^2_y}\right) \left\|J_y^{s_2}u\right\|_{L^2_y}dx \\
	&\leq \left\| \partial_y
          u\right\|_{L_{xy}^\infty}\left\|J_y^{s_2-1}\partial_x
          u\right\|_{L^2_{xy}}\left\|J_y^{s_2}u\right\|_{L^2_{xy}}+\left\|
          \partial_x
          u\right\|_{L_{xy}^\infty}\left\|J_y^{s_2}u\right\|_{L^2_{xy}}^2 \\
	&\leq \left\| \partial_y
          u\right\|_{L_{xy}^\infty}\left\|J_x^{s_2}
          u\right\|_{L^2_{xy}}^\frac1{s_2}\left\|J_y^{s_2}u\right\|^\frac{2s_2-1}
        {s_2}_{L^2_{xy}}+ \left\| \partial_x u\right\|_{L_{xy}^\infty}\left\|J_y^{s_2}u\right\|_{L^2_{xy}}^2\\
&\lesssim \left(\left\| \partial_x u\right\|_{L_{xy}^\infty}+\left\| \partial_y u\right\|_{L_{xy}^\infty}\right)\left\| u\right\|_{H^{s_1,s_2}_{xy}}^2.
\end{split}
\end{equation}
So, integrating by parts in the second term on the right side of the inequality
\eqref{eqestenergia11} and from last inequality, 
we have
\begin{equation}\label{eqestenergia15}
	\frac{d}{dt}\left\|J_y^{s_2}u\right\|^2_{L^2_{xy}}\lesssim   \left(\left\| \partial_x u\right\|_{L_{xy}^\infty}+\left\| \partial_y u\right\|_{L_{xy}^\infty}\right)\left\| u\right\|_{H^{s_1,s_2}_{xy}}^2.
\end{equation}
Gathering  \eqref{eqestenergia10} and
\eqref{eqestenergia15}, we obtain  
\begin{equation}\label{eqestenergia16}
  \frac{d}{dt}\left\|u\right\|^2_{H^{s_1,s_2}_{xy}}\lesssim   \left(\left\| \partial_x u\right\|_{L_{xy}^\infty}+\left\| \partial_y u\right\|_{L_{xy}^\infty}\right)\left\| u\right\|_{H^{s_1,s_2}_{xy}}^2.
\end{equation}
From Gronwall inequality we get the lemma.
\end{proof}
\end{lema}
\subsection{A Strichartz estimate type}
The energy estimate suggests that we have to control the norms
$\left\|\partial_xu\right\|_{L_T^1L_{xy}^\infty}$ and
$\left\|\partial_yu\right\|_{L_T^1L_{xy}^\infty}$ in order to prove
the local well-posedness of the Cauchy problem
\eqref{eq:principal}. The next lemma shows how we can control these
norms.
\begin{lema}\label{lemaestnorstr}
  Let $-1\leq \alpha \leq 1$, $T>0$ and assume that
  $u\in C([0,T];H^\infty(\re^2))$ is a solution to initial value
  problem \eqref{eq:principal} with initial condition $\psi$.  Then,
  for any $s_1>\frac{17}{12}-\frac{\alpha}{4}$ such that
  $$ \begin{cases}
    s_2>1,              & \mbox{if } 0\leq\alpha\leq 1,   \\
    \frac{1}{s_2}-\frac{\alpha}{4s_1}<1, & \mbox{if } -1\leq\alpha\leq
    0
  \end{cases} $$ and $s_1\geq s_2$, there exist constants
  $C_{s_1,s_2}$ and $k_{s_1,s_2}\in(7/12,1)$ such that
\begin{equation}\label{eqestnorstr}
f(T)=\left\|u\right\|_{L_T^1L_{xy}^\infty}+\left\|\partial_xu\right\|_{L_T^1L_{xy}^\infty}+\left\|\partial_yu\right\|_{L_T^1L_{xy}^\infty}
\end{equation}
satisfies
\begin{equation}\label{eqestnorstr2}
f(T)\leq C_{s_1,s_2} T^{k_{s_1,s_2}} (1+f(T)) \left\|u\right\|_{L_T ^\infty H^{s_1,s_2}_{ xy}}.
\end{equation}
\begin{proof}
  First let us examine an estimate for
  $\left\|\partial_x u\right\|_{L^1_TL_{ xy}^ \infty}$. From
  the refined Strichartz estimate \eqref{eqstrichartzref}, 
  we have 
\begin{equation}\label{eqestnorstr3}
\left\|\partial_x u\right\|_{L^1_TL_{xy}^\infty}\leq c_\epsilon T^\frac{7+2\epsilon}{12} \left(\sup_{t\in[0,T]}{\left\|J_x^{\frac{17}{12}-\frac{\alpha}{4}+\frac{2\epsilon}{3}}u\right\|_{L^2_{xy}}}+\int_0^T \left\|D_x^{\frac{5}{12}-\frac{\alpha}{4}+\frac{2\epsilon}{3}}(u\partial_x u)\right\|_{L^2_{xy}}dt\right).
\end{equation}
Let us choose  $0<\epsilon_1<1$ in such a way that $\frac{5}{12}- \frac{
  \alpha}{4}+\frac{ 2\epsilon_1}{3}<1$ and
${\frac{17}{12}-\frac{\alpha}{4}+\frac{2 \epsilon_1}{3}}<s_1$. So,
the first term from the right side of \eqref{eqestnorstr4} we have
\begin{equation}\label{eqestnorstr4}
\sup_{t\in[0,T]}{\left\|J_x^{\frac{ 17}{ 12}- \frac{ \alpha }{ 4}+
      \frac{ 2\epsilon_1}{3}}u \right \|_{L^2_{ xy}}}\leq \sup_{ t\in
  [0,T]}{ \left\|J_x^{s_1} u\right\|_{L^2_{xy}}}\leq \left\| u\right
\|_{ L^\infty_TH^{s_1,s_2}_{xy}}.
\end{equation}
For the second term on the right side of \eqref{eqestnorstr4},
from the Leibniz rule (equation \eqref{eqreglaleibniz}) in $x$,
we have
\begin{equation}\label{eqestnorstr5}
\begin{split}
&\int_0^T
\left\|D_x^{\frac{5}{12}- \frac{ \alpha}{4}+ \frac{ 2\epsilon_1}{ 3}}(
  u\partial_xu)\right\|_{L^2_{xy}}dt\lesssim
 \\ &\lesssim \int_0^T {\left\|{{ \left\|D_x ^{\frac{ 5}{12}- \frac{
               \alpha}{4}+ \frac{2\epsilon_1}{3 }}u\right\|_{
           L^2_{x}}}{ \left\|\partial_x u\right \|_{L_{ x}^\infty}}+{
         \left\| D_x^{ \frac{5}{12}-\frac{\alpha}{4}+\frac{2 \epsilon_1}{3}}\partial_xu\right\|_{L^2_{x}}}{\left\|u\right\|_{L_{x}^\infty}}}\right\|_{L^2_y}} dt\\
&\lesssim\int_0^T {\left\|{{\left\|J_x^{s_1}u\right\|_{L^2_{x}}}{\left\|\partial_x u\right\|_{L_{x}^\infty}}+{\left\|J_x^{s_1}u\right\|_{L^2_{x}}}{\left\|u\right\|_{L_{x}^\infty}}}\right\|_{L^2_y}} dt\\
&\lesssim \int_0^T \left(\left\|u\right\|_{L^\infty_{xy}}+\left\|\partial_xu\right\|_{L^\infty_{xy}}\right) \left\|J_x^{s_1}u\right\|_{L^2_{xy}}dt\\
&\lesssim \left\|u\right\|_{L_T^\infty H^{s_1,s_2}_{xy}}f(T)
\end{split}
\end{equation}
Then, from \eqref{eqestnorstr5}, \eqref{eqestnorstr4} and
\eqref{eqestnorstr3}
\begin{equation}\label{eqestnorstr6}
\left\|\partial_x u\right\|_{L^1_TL_{xy}^\infty}\leq
C_1(s_1)T^\frac{7+ 2\epsilon_1}{12} (1+ f(T))\left \|u \right\|_{L_T^
  \infty H^{ s_1,s_2}_{xy}}.
\end{equation}
Let us estimate now $\left\|u\right\|_{L_T^1L_{xy}^\infty}$. From 
Duhamel's principle, Cauchy-Schwarz inequality in $t$ and Corollary
\ref{coronormalt2lxyinf}, we have that
\begin{equation}\label{eqestnorstr7}
\begin{split}
\left\|u\right\|_{L_T^1L_{xy}^\infty}\leq \left\|W_\alpha (t)\psi\right\|_{L_T^1L_{xy}^\infty}+\int_0^T\left\|W_\alpha(t-t')(u\partial_xu)(t')\right\|_{L_T^1L_{xy}^\infty}dt'\\
\leq T^\frac{1}{2} \left(\left\|W_\alpha
    (t)\psi\right\|_{L_T^2L_{xy}^\infty}+\int_0^T
  \left\|W_\alpha(t-t')(u\partial_xu)(t')\right\|_{L_T^2L_{xy}^\infty}dt'\right)\\
\leq C_2 T^\frac{7+2\epsilon_2 }{12}
\left(\left\|D_x^{\frac{1}{2}(\epsilon_2-\frac{\alpha}{2})}\psi\right\|_{L_{xy}^2}+\int_0^T
  \left\|D_x^{\frac{1}{2}(\epsilon_2-\frac{\alpha}{2})}(u \partial_x
    u)( t')\right\|_{L_{xy}^2}dt'\right), 
\end{split}
\end{equation}
for some $\epsilon_2=\epsilon_2(s_1)$ such that
$0<\frac{1}{ 2}( \epsilon_2 -\frac{\alpha}{2})$ and
$\frac{ 1}{ 2}( \epsilon_2 -\frac { \alpha}{ 2})+1<s_1$.  The first
term in the last line of equation above,  thanks to the
choice of $\epsilon_2$, is less than $\left\| J_x^{s_1} \psi\right
\|_{ L^2_{xy}}$. For the term within the integral, from the 
Leibniz rule in the $x$ variable, we have
\begin{equation}\label{eqestnorstr9}
\begin{split}
\left\|D_x^{{\frac{1}{2}(\epsilon_2-\frac{\alpha}{2})}}(u\partial_xu)\right\|_{L_{xy}^2}
&\leq
\left\|D_x^{\frac{1}{2}(\epsilon_2- \frac{ \alpha}{ 2})} u
\right\|_{L_{ xy}^2}\left\|\partial_xu\right \|_{L_{ xy}^ \infty}
+\left\| D_x^{\frac{1}{ 2}( \epsilon_2- \frac{ \alpha}{2})} \partial_x
  u\right\|_{L_{xy}^2}\left\|u\right\|_{L_{xy}^\infty}\\
&\leq \left\|J_x^{s_1}u\right\|_{L_{ xy}^ 2} \left\| \partial_x u
\right \|_{L_{ xy}^\infty}+ \left\|J_x^ {s_1}u \right\|_{L_{ xy}^2}
\left\| u\right\|_{L_{xy}^\infty}.
\end{split}
\end{equation}
So that
\begin{equation}\label{eqestnorstr10}
\begin{split}
\left\|u\right\|_{L_T^1L_{xy}^\infty}
&\leq C_2 T^{\frac{7+2\epsilon_2}{12}} \left( \left\|J_x^{s_1}\psi\right\|_{L^2_{xy}} +\int_0^T \left( \left\|u\right\|_{L_{xy}^\infty}+\left\|\partial_xu\right\|_{L_{xy}^\infty} \right) \left\|J_x^{s_1}u\right\|_{L_{xy}^2}dt'\right)\\
&\leq C_2 T^{\frac{7+2\epsilon_2}{12}} \left( \left\|J_x^{s_1}\psi\right\|_{L^2_{xy}} +\int_0^T \left( \left\|u\right\|_{L_{xy}^\infty}+\left\|\partial_xu\right\|_{L_{xy}^\infty} \right) \left\|J_x^{s_1}u\right\|_{L_{xy}^2}dt'\right)\\
&\leq C_2 T^{\frac{7+2\epsilon_2}{12}} (1+f(T))\left\|u\right\|_{L_T^\infty H^{s_1,s_2}_{xy}}
\end{split}
\end{equation}
where $C_2=C_2(s_1)$.
Lastly, let us estimate $\left\|\partial_y
  u\right\|_{L_T^1L_{xy}^\infty}$. In the same way as we estimate the
other two terms, we have  that
\begin{equation}\label{eqestnorstr11}
\begin{split}
\left\|\partial_y u\right\|_{L_T^1L_{xy}^\infty}
\leq C_3 T^{\frac{ 7 + 2\epsilon_3} {12}} \left(\left\|D_x^{\frac{1}{2}(\epsilon_3-\frac{\alpha}{2})}\partial_y\psi\right\|_{L_{xy}^2}+\int_0^T\left\|D_x^{\frac{1}{2}(\epsilon_3-\frac{\alpha}{2})}\partial_y (u\partial_xu)\right\|_{L_{xy}^2}dt\right),
\end{split}
\end{equation}
where we will choose $\epsilon_3$ later. For the term within
the integral, from the commutator estimate (Lemma
\ref{lemaconmkatoponce2}) in the $x$ variable, we obtain
\begin{equation}\label{eqestnorstr12}
\begin{split}
\left\|D_x^{\frac{1}{2}(\epsilon_3-\frac{\alpha}{2})+1} (u\partial_yu)\right\|_{L_{xy}^2}\leq \left\|\left[D_x^{\frac{1}{2}(\epsilon_3-\frac{\alpha}{2})+1},u\right]\partial_y u\right\|_{L_{xy}^2}+\left\|uD_x^{\frac{1}{2}(\epsilon_3-\frac{\alpha}{2})}\partial_y u\right\|_{L_{xy}^2}\\
\leq \left(\left\|u\right\|_{L_{xy}^\infty}+\left\|\partial_xu\right\|_{L_{xy}^\infty}\right)\left\|D_x^{\frac{1}{2}(\epsilon_3-\frac{\alpha}{2})} \partial_y u\right\|_{L_{xy}^2}+\left\|\partial_yu\right\|_{L_{xy}^\infty}\left\|D_x^{\frac{1}{2}(\epsilon_3-\frac{\alpha}{2})+1} u\right\|_{L_{xy}^2}.
\end{split}
\end{equation}
Let us, then, examine 
$\left\|D_x^{\frac{1}{ 2}( \epsilon_3- \frac{ \alpha}{ 2}) +1}
  u\right\|_{L_{xy}^2}$ and
$\left\|D_x^{\frac{1}{2}( \epsilon_3- \frac{ \alpha}{ 2})} \partial_y
  u\right\|_{L_{xy}^2}$.  We will do this for two cases in $\alpha$.
The first case is when $-1\leq\alpha\leq 0$. In this case
$\frac{1}{2}(\epsilon_3-\frac{\alpha}{2})>0$ for
$0<\epsilon_3<1$. Choose $\epsilon_3$ in such a way that
$\frac{1}{2}(\epsilon_3-\frac{\alpha}{2})+1<s_1$ and
$\frac{\epsilon_3}{2s_1}-\frac{\alpha}{4s_1}+\frac{1}{s_2}<1$, or
equivalently
$\frac{1 }{2 }( \epsilon_3- \frac{ \alpha}{ 2}) \frac{ s_2}{ s_2
  -1}<s_1$. So,
\begin{equation}\label{eqestnorstr13}
	\left\|D_x^{\frac{1}{2}(\epsilon_3-\frac{\alpha}{2})+1}
          u\right\|_{L_{xy}^2}\leq \|J_x^{s_1}u\|_{L^2_{x,y}}\leq
        \|u\|_{H^{x_1. x_2}_{x,y}}  
      \end{equation}
      and
    \begin{equation}\label{eqestnorstr14}
  \begin{split}
	\left\|D_x^{\frac{1}{2}(\epsilon_3-\frac{\alpha}{2})}
 \partial_y u\right\|_{L_{xy}^2} &\leq \left\|J_x^{\frac{1}{2}(\epsilon_3-\frac{\alpha}{2})}J_yu\right\|_{L_{xy}^2}
\leq \left\|J_x^{\frac{1}{2}(\epsilon_3-\frac{\alpha}{2})\frac
    {s_2 }{s_2- 1}}u\right\|_{L_{xy}^2}^\frac{s_2-1}{ s_2} \left\|
  J_y^{ s_2}u\right\|_{L_{xy}^2}^\frac1{s_2}  \\ &\le
\left\|J_x^{s_1}u\right\|_{L_{xy}^2}^\frac{s_2-1}{ s_2} \left\| J_y^{
    s_2}u\right\|_{L_{xy}^2}^\frac1{ s_2}\le  \|u\|_{H_{x,y}^{s_1,s_2}}
\end{split}
\end{equation}
The second case is when $0\leq\alpha\leq 1$. Choose, then,
$\epsilon_3>\frac{\alpha}{2 }$ such that also
$\frac{1}{2}(\epsilon_3-\frac{\alpha}{2})+1<s_1$ and
$\frac{\epsilon_3}{2s_1}-\frac{\alpha}{4s_1}+\frac{1}{s_2}<1$. So we are
in the same situation of the first case, what leads us to the same 
two inequalities \eqref{eqestnorstr13} and
\eqref{eqestnorstr14}. This inequalities together to
\eqref{eqestnorstr12} and \eqref{eqestnorstr11} allow us show that 
\begin{equation}\label{eqestnorstr15}
\begin{split}
\left\|\partial_y u\right\|_{L_T^1L_{xy}^\infty}
\leq C_3 T^\frac{7+2\epsilon_3}{12} \left( \left \|\psi \right \|_{
    H_{xy}^{s_1,  s_2}}+(1+ f(T))\left\|u\right\|_{L_{T}^{\infty}
    H_{xy}^{s_1,s_2}}\right). 
\end{split}
\end{equation}
The theorem follows from (\ref{eqestnorstr6}), (\ref{eqestnorstr10}) and
(\ref{eqestnorstr15}), taking $C_{s_1,s_2}=C_1+C_2+C_3$ and
$k_{s_1,s_2}=\max_{i= 1,2,3}{\frac{7+2\epsilon_i}{12}}$.
\end{proof}
\end{lema}

\subsection{Local well-posedness} 
We already have all the necessary ingredients to prove the following
theorem. 
\begin{teo}\label{bplxs}
Let $-1<\alpha<1$ and $0<\epsilon<1$ be such that
\begin{equation*}
	s_1> \frac{17+ 4\epsilon -3\alpha } {12} ,\quad s_2\leq s_1
        \text{ and} \quad \begin{cases}
          s_2>1, & \text{ if } \alpha>0,\\
          \frac{5+ 4\epsilon -3\alpha}{12  s_1}+\frac{1}{s_2}<1, &
          \text{ if } \alpha\le 0. 
        \end{cases}
      \end{equation*}
Then, for any $\psi\in H^{s_1,s_2}(\re^2)$, there exist a
time $T=T(\left\|\psi\right\|_{H^{s_1,s_2}})$ and a unique solution
$u\in C([0,T]:H^{s_1,s_2})$ to the Cauchy problem
\eqref{eq:principal} such that $u, \partial_x u, \partial_y u \in L_T^1
L_{xy}^\infty$. Furthermore, if $0<T'<T$, there exists a neighborhood
$\mathcal{V}$ of $\psi$ in $H^{s_1,s_2}(\re^2)$ such that $\psi\mapsto u(t)$ is
continuous
\end{teo} 
From the local well-posedness of \eqref{eq:principal} in $H^{\sigma}$
for $\sigma>2$ ($\sigma$ fix), for $\psi \in H^\infty(\re^2) $, there
exists a unique solution $u\in C([0,T^*]:H^{\sigma}(\re^2))$,
where $T^*$ is the maximal existence time of the solution, such
that if $T^*<\infty$, then
\begin{equation}\label{abu}
	\lim_{t\nearrow T^*} \left\|u(t)\right\|_{H^{\sigma}}=\infty .
\end{equation}
Thanks to the above we have the following a priori estimate  that
will be useful in the proof of the existence of the solution to
problem \eqref{eq:principal}.
\begin{lema}\label{lemaext}
  Let $2<\sigma$, $-1\leq \alpha\leq 1$ and $s_1$, $s_2$ be as in Theorem
  \eqref{bplxs} such that $s_i\le \sigma$, $i=1, 2$. Let, also,
  $\psi\in H^{\sigma}$ and $u\in C([0,T^*);H^{\sigma}(\re^2))$ such that
  $u(0)=\psi$, where $T^*$ is the maximal time of existence of
  $u$. Then, there exist $K_0=K_0(s_1,s_2)>0$ and $L_{s_1,s_2}>0$ such that
  $T^*>T$, where
  $(L_{s_1,s_2} \left\|\psi\right\|_{H^{s_1,s_2}}+1)^\frac{12}{7}T=1$,
  and
\begin{equation}\label{ext}
\begin{split}
\left\|u\right\|_{L_T^\infty H^{s_1,s_2}} &\leq 2 \left\|\psi\right\|_{H^{s_1,s_2}}\\
f(T)= \left\|u\right\|_{L_T^1 L_{xy}^\infty} + \left\| \partial_x
  u\right \|&_{L_T^1 L_{xy}^\infty} + \left\| \partial_y u\right\|_{
  L_T^1 L_{xy}^ \infty}\leq K_0
\end{split}
\end{equation}
\begin{proof}
For $s_1$ and $s_2$ as in the statement of lemma, let 
\begin{equation}\label{ext1}
	T_0=\sup_{\tilde T\in (0,T^*)} \left\{ \tilde T\, |\, \left\| u \right\|_{
            L_{\tilde T}^ \infty H^{s_1,s_2}}\leq
          2\left\|\psi\right\|_{H^{s_1,s_2}}  \right\}
\end{equation}
From the local well-posedness, this set is not empty. Let $C_{s_1,s_2}$ be
as in Lemma \ref{lemaestnorstr}, $C$ be as in Lemma
\ref{lemaestenergia}, $L_{s_1, s_2}=2(2C-1)C_{s_1, s_2} $
and $$T=\frac 1{ (L_{s_1, s_2}\|\psi\|_{H^{s_1,s_2}}
  +1)^\frac{12}7}. $$ Let us see that $T\le T_0$. Suppose not. Thanks
to Lemma \ref{lemaestnorstr}
\begin{equation}\label{ext2}
	f(T_0)\leq C_{s_1,s_2} T_0^{k_{s_1,s_2}} (1+f(T_0))
        \left\|u\right\|_{L_{T_0}^\infty H^{s_1,s_2}_{xy}}\leq
        C_{s_1,s_2} T_0^{\frac{7}{12}} (1+ f(T_0))\left\|u\right\|_{L_{T_0
          }^\infty H^{ s_1,s_2}_{xy}} 
\end{equation}
Since $\left\|u\right\|_{L_{T_0} ^\infty H^{s_1,s_2}}\leq
2\left\|\psi\right\|_{H^{s_1,s_2}}$ and $T_0<T$ ,
\begin{equation*}
	f(T_0)\leq 2 C_{s_1,s_2}\frac{(1+ f(T_0 ))}{ L_s \left \|\psi
          \right \|_{H^{s_1,s_2}}+1}\left\|\psi\right\|_{H^{s_1,s_2}}
\end{equation*}
or equivalently
\begin{equation*}
	(4CC_{s_1,s_2}\left\|\psi\right\|_{H^{s_1,s_2}}+1)f(T_0)\leq 2 C_{s_1,s_2} \left\|\psi\right\|_{H^{s_1,s_2}}
\end{equation*}
Therefore,
\begin{equation*}
	f(T_0)\leq \frac 1{2C}
\end{equation*}
By Lemma \ref{lemaestenergia} we would have
$$ \|u(T_0)\|_{H^{s_1,s_2}}\le e^{\frac12} \|\psi\|_{H^{s_1,s_2}} <2
\|\psi\|_{H^{s_1,s_2}}. $$ From the continuity of $u$, it follows that there
exists $\tilde T>T_0$ such that
$$ \|u\|_{L^\infty_{\tilde T}H^{s_1,s_2}}\le 2
\|\psi\|_{H^{s_1,s_2}},$$ which contradicts the choice of $T_0$. Then,
$T\le T_0$. In particular, 
$$ \|u\|_{L^\infty_{T}H^{s_1,s_2}}\le 2
\|\psi\|_{H^{s_1,s_2}},$$
and, repeating the above reasoning, from inequality \eqref{ext2},
taking $T$ instead of $T_0$ we have that \begin{equation*}
	f(T)\leq \frac 1{2C}.
\end{equation*}
This proves the lemma.
\end{proof}
\end{lema}
\begin{corol}
  Let $\psi$ and $T$ be as in the last lemma. If $\psi\in H^\infty$,
  then the solution $u$ to the Cauchy problem
  \eqref{eq:principal}, with $u(0)= \psi$, belongs to the set
  $C([0,T]; H^\infty)$.
\end{corol}
\begin{proof}
  Let $T$ be as in the lemma above. From that lemma, for any 
  number $\sigma$ that satisfies the condition established there, we
  have 
  $u\in C([0,T]; H^{\sigma})$. From here it follows the corollary.
\end{proof}
\begin{corol}\label{eq:bp5}
  Let $R>0$ and $\psi\in H^\infty$ be such that $\|\psi\|_{H^{ s_1, s_2}}\le
  R$. Then, there exists $T_0$ that depends on $R$ and $M$, constant that
  depends only on $s_1, s_2$, such that $u\in C([0,T_0]; H^\infty)$
  and $$ f(T_0)= \left\|u\right\|_{L_{T_0}^1 L_{xy}^\infty} + \left\| \partial_x
  u\right \|_{L_{T_0}^1 L_{xy}^\infty} + \left\| \partial_y u\right\|_{
  L_{T_0}^1 L_{xy}^ \infty}\leq M.$$ 
\end{corol}
\begin{proof}
  If we set $T_0= 1/ {(L_{s_1, s_2}R+1 )^\frac{12}7 }$, we have
  that $T_0\le T$, $T$ as in the last lemma. From the proof of
  that lemma it follows that  $f(T_0)\le 1/2C$, that does not depend on initial
  data, only on $s_1$ and $s_2$. Making $M=1/2C$ shows the corollary. 
\end{proof}
\subsubsection{Proof of Theorem \ref{bplxs}}
\begin{lema}\label{lem:L2lip}
  Suppose that $\psi$ and $\phi\in H^\infty$ and that $u$ and
  $v\in C([0,T]; H^\infty)$ are the solutions to the problem
  \eqref{eq:principal} with initial conditions $\psi$ and $\phi$,
  respectively. Then, $$\|u- v\|_{L^2}(T_0(R)) \le \|\psi
  - \phi \|_{L^2} e^{CM},$$ where $T_0(R)$ and $M$ are as in
  Corollary \ref{eq:bp5}, $C$ is a constant that depends only on $s_1$
  and $s_2$, and $R$ is the maximum between the norm of $\phi$ and
  $\psi$ in the space $H^{ s_1, s_2}$.  
\end{lema}
\begin{proof}
  The proof is analogous to obtaining the energy estimate. Let us
  see. Let $u$ and $v$ be as in statement of lemma.
  Then, 
  $$ \partial_t (u - v) = \partial_x^3 (u-v) - \mathcal H D^\alpha
  \partial_y^2 (u-v) + \frac12 (u+ v)\partial_x (u -v) + \frac12 (u-
  v)\partial_x (u +v).$$ Multiplying by $u-v$ on both sides the
  equation and integrating by parts we have that
  $$ \frac12\frac d{dt} \|u -v \|_2^2 \le \frac14 (\|u_x\|_{L^\infty }
  +\|v_x\|_{L^\infty }) \|u -v\|_2^2 .$$ From the Gronwall Lemma and
  Corollary \ref{eq:bp5} it follows the lemma.  
\end{proof}
Now, let $\psi\in H^{s_1, s_2}$ and suppose that $\psi_n$ is a
sequence of functions in $H^\infty$ that converges to $\psi$ in
$H^{s_1, s_2}$ . Taking $R=\sup_n \|\psi_n \|$, from lemma above, the
solutions to \eqref{eq:principal}
$u_n\in C([0, T_0];H^{\infty}(\re^2))$, with initial condition
$\psi_n$, converge uniformly to a function $u$ in
$C([0, T_0];L^2(\re^2))$. Moreover, from Corollary \ref{eq:bp5}, the
functions $u_n$ are uniformly bounded in $H^{s_1, s_2}$. Therefore,
from Banach-Alaoglu theorem, $u_n(t)$ have a subsequence that
converges weakly in $H^{s_1, s_2}$. From the uniform convergence of
$u_n(t)$ to $u(t)$ in $L^2$, it follows that  $u(t) \in H^{s_1, s_2}$, for
each $t\in [0,T_0]$. From the continuity of $u$ from $[0, T_0]$ to $L^2$
follows the weak continuity of $u$ from $[0, T_0]$ to
$H^{s_1, s_2}$. In particular, from uniform boundedness of the
sequence $u_n(t)$ and Lemma \ref{interdsenlplq}, it follows that $u_n$
converges strong and uniformly in $H^{s'_1,s'_2}$ to $u$, for
any pair of non negative real numbers  $s'_1, s'_2$
strictly less than $s_1, s_2$, respectively.  Since each one of the
functions $u_n$ satisfies the integral equation associated to
\eqref{eq:principal},
\begin{equation} \label{eq:int} w = W_\alpha(t)w(0) + \frac12 \int_0^t W_\alpha(t-t') \partial_x
(w^2(t'))\, dt', 
\end{equation}
in $H^{s'_1-1,s'_2-1}$, it follows that function $u$ satisfies this
same integral equation in $H^{s'_1-1,s'_2-1}$, $1< s'_2<s_2$. From
Lemma \ref{lemastrichartzref},  considering $17/ 12 - \alpha /4<
s_1'<s_1$, follows that $\partial_x u \in {L^1_T L^\infty _{x,
    y}}$. In the same way, thanks to the inequalities
\eqref{eqestnorstr7} and \eqref{eqestnorstr11}, we obtain that
$u$ and $\partial_y u \in {L^1_T L^\infty _{x, y}}$. So,
\begin{lema}
  Let $\psi\in H^{s_1, s_2}(\re ^2)$. Then, there exists $T>0$ and a 
  unique $u\in C([0, T];L^2(\re^2)) \cap C_w([0, T];H^ {s_1, s_2}(\re^2))\cap
  C^1([0, T];H^{-3} \cap (X^3)*)$ solution to problem
  \eqref{eq:principal}. Furthermore, $u$,
  $\partial_x u$, $\partial_y u \in {L^1_T L^\infty _{x, y}}$ and the
  map $\psi \mapsto u $ is Lipschitz continuous from
  $L^2(\re^2)$ to $C([0, T];L^2(\re^2)) $.
\end{lema}
\begin{proof}
  It only remains to prove that $u$ is unique and that the map
  $\psi \mapsto u $ is Lipschitz continuous from $L^2(\re^2)$ to
  $C([0, T];L^2(\re^2)) $. Let $u$ and $v$ be as in the statement of
  lemma. Since $u$ and $v$ are solutions to the integral equation,
  from Lemma \ref{lemastrichartzref}, $u_x$ and
  $v_x \in {L^1_T L^\infty _{x, y}}$. Now proceed as in the lemma
  above, but we need Lemma \eqref{justpin}, to obtain that
  $$ \frac12 \frac d {dt} \|u -v\|_{L^2}^2 = \frac 14 (\partial_x u
  +\partial_x v, (u-v)^2)_{L^2 }\le (\|u_x\|_{L_{x,y}^\infty }
  +\|v_x\|_{L_{x,y}^\infty }) \|u -v\|_{L^2}^2 . $$ Since
  $\|u_x\|_{L_{x,y}^\infty } +\|v_x\|_{L_{x,y}^\infty }$ is integrable
  and $\|u -v\|_{L^2}^2$ is continuous, from Gronwall lemma, it follows the
  theorem.
\end{proof}
\begin{lema}
  For $\psi\in H^{s_1, s_2} $ as in the lemma above, the solution $u$ to
  \eqref{eq:principal} described there is strongly continuous in
  $H^{s_1, s_2}$
\end{lema}
\begin{proof}
  Let $\psi$, $\psi_n$, $u_n$, $n=1, 2, 3, \cdots$, and $u$ be as
  before. Let, also, $M$, $T_0$ be as in Corollary \ref{eq:bp5}
  and $T\le T_0$. It is clear that
  $f_n(T)=\|u_n\|_{L_{x,y}^ \infty}+ \|\partial_x u_n\|_{L_{x,y}^
    \infty}+ \|\partial_y u_n\|_{L_{x,y}^ \infty} $ satisfies
  that $$f_n(T)\le M,$$ for all $n$. From Lemma \ref{lemaestnorstr} and
  Corollary \ref{eq:bp5}, it follows that $$ f_n(T)\le 2C_{s_1, s_2}
  T^{k_{s_1, s_2}}  +(1+ M) R.$$ From Lemma \ref{lemaestenergia},
  $$ \|u_n(T ) \|_{H^{s_1, s_2}} \le\| \psi_n\|_{H^{s_1, s_2}}
  e^{2C_{s_1, s_2} T^{k_{s_1, s_2}} (1+ M) R},$$ for all $n$. From the
  weak convergence of $u_n(T)$ to $u(T)$ in $H^{s_1, s_2}$ and since
  $\psi_n$ converges to $\psi$ in $H^{s_1, s_2}$, it follows that
  $$ \|u(T )\|_{H^{s_1, s_2}}\le \| \psi\|_{H^{s_1, s_2}} e^{2C_{s_1, s_2} T^{k_{s_1,
        s_2}} (1+ M) R}.$$ From this last inequality and the weak continuity
  of $u$ in $H^{s_1, s_2}$, it follows that
  $$\| \psi\|_{H^{s_1, s_2}}\le \liminf_{T\to 0+} \|u(T )\|_{H^{s_1,
      s_2}}\le \limsup_{T\to 0+} \|u(T )\|_{H^{s_1, s_2}} \le \|
  \psi\|_{H^{s_1, s_2}} . $$ Then, $u$ is right strongly continuous
  at $0$ in the space $H^{s_1, s_2}$. Since $u(-t, -x, -y)$ is
  also solution to the equation \eqref{eq:principal}, we have that 
  $u$ is also left strongly continuous at $0$ in the
  space $H^{s_1, s_2}$. Now, for any $t^*\in [0, T]$,
  $u(t+t^*)$ is also solution to the equation \eqref{eq:principal}
  with initial condition $u(t^*)$. Then from the unicity of the solution,
  $u$ also is strongly continuous at $t^*$ in the space
  $H^{s_1, s_2}$. This ends the proof.
\end{proof}
Now examine the continuity of the solutions to \eqref{eq:principal}
with respect to initial data. For this purpose we will use a technique
very useful an recurrently used in the literature related to the
well-posedness of evolution equations. This technique is the
Bona-Smith method of approximation introduced in \cite{bona-smith}. Let
us see.
\begin{lema} \label{prop9} Let $\phi \in H^{s_1, s_2}$, $s_1$ and
  $s_2$ be positive real numbers. For each $\tau>0$ define $\phi^\tau$
  by
   \begin{equation}
     \phi^\tau(x,y ) = \left( \widehat{\phi}(\xi, \eta) \exp 
         \left( -\tau \left(( 1 + \lvert \xi \rvert ^2
             )^{\frac{s_1}{2}} + ( 1 +
         \lvert \eta \rvert ^2 )^{\frac{s_2}{2}}  \right) \right)  \right)
     \text{\begin{huge}$ \check{ \ }$\end{huge}} (x,y).  \label{aprox1} 
\end{equation}
Then \[ \lim_{\tau \to 0+} \|\phi^\tau - \phi\|_{H^{s_1,s_2}}=0 \] and
there exists a constant $C=C(s)$ such that
\begin{align}
  \| \phi^\tau\|_{H^{s_1+1, s_2}}  &\leq C  \left( \frac{1}{\tau s_1}
  \right) ^{\frac{1}{s_1}}  \| \phi \|_{H^{s_1,s_2}}, \label{eq:2.49}\\
  \| \phi^\tau\|_{H^{s_1, s_2+1}}  &\leq C  \left( \frac{1}{\tau s_2}
  \right) ^{\frac{1}{s_2}}  \| \phi \|_{H^{s_1,s_2}} \label{eq:2.50}
\end{align}
and
\begin{equation}\label{eq:2.51}
  \|\phi^\tau - \phi^\theta\|_{L^2} \leq C \lvert \tau - \theta
  \rvert  \|\phi\|_{H^{s_1, s_2}}
\end{equation}
\end{lema}
\begin{prop} \label{prop10} Let $R>0$, and assume that $\Lambda$ is a set and
  $\psi_{\lambda\in \Lambda}$ is a collection of functions in $H^\infty$
  such that  $ \| \psi_{\lambda}\|_{ H^{s_1, s_2}}\le R$, for all
  $\lambda \in \Lambda $. Also, let $\psi_\lambda^\tau$ be the approximations
  defined from $\psi_\lambda$ as in \eqref{aprox1}, and assume that
  $u_{\lambda}^\tau$ is the solution to \eqref{eq:principal} with
  initial condition $\psi_\lambda^\tau$, for all $\lambda \in
  \Lambda$. Suppose that $0 \leq \theta < \tau$. Then, for $\nu>0$
  $$ 
  \nsc{u_\lambda^\tau - u_\lambda^\theta}{H^{s_1, s_2}} \leq C \left(
    \nsc{\psi_\lambda^\tau - \psi_\lambda^\theta}{H^{s_1, s_2}} + \tau^\nu \right).
$$ 
for all $\lambda\in \Lambda$.
\end{prop}

\begin{proof}
  It is evident that $u_\lambda^\tau(t)$ also is defined on
  $[0,T_0]$, for all $n$ and $\tau$. We will proceed as in the
  proof of Lemma \ref{lemaestenergia}. Then,
\begin{equation}\label{eq:A1.0}
  \begin{aligned}
    \frac{1}{2} \frac{d}{dt} \| J_x^{s_1}
    (u_\lambda^\tau& - u_\lambda^\theta) \|_{L^2}^2 = \int_{\re^2} J_x^{s_1}
     (u_\lambda^{\tau } \partial_x u_\lambda^\tau - u_\lambda^{\theta
     } \partial_x u_\lambda^\theta) J_x^{s_1} 
    ( u_\lambda^\tau - u_\lambda^\theta)
    \, dxdy \\ =&\frac12 \int_{\re^2} J^{s_1}_x \partial_x\left( (u_\lambda^{\tau }
      + u_\lambda^\theta) ( u_\lambda^{\tau } - u_\lambda^\theta) \right) J_x^{s_1}
    ( u_\lambda^\tau - u_\lambda^\theta) \, dxdy \\ =& \frac12
    \int_{\re^2} J^{s_1}_x \left ((u_\lambda^{\tau } +
      u_\lambda^\theta)\partial_x ( u_\lambda^{\tau } - u_\lambda^\theta) \right)
    J^{s_1}_x ( u_\lambda^\tau - u_\lambda^\theta) \,
    dxdy + \\ &+\frac12 \int_{\re^2} 
     J_x^{s_1}  \left( \partial_x( u_\lambda^{\tau } + u_\lambda^\theta) ( u_\lambda^{\tau } -
      u_\lambda^\theta) \right) J^{s_1}(
    u_\lambda^\tau - u_\lambda^\theta) \, dxdy.
  \end{aligned}
\end{equation}
We estimate the last two terms that appear in the inequality above. From
Lemma \ref{katoponce}, for the first term, we have
\begin{equation}\label{eq:A1.0.1}
\begin{aligned}
\frac12 \int_{\re^2} 
  J^{s_1}_x\big( (u_\lambda^{\tau } &+ u_\lambda^\theta) \partial_x (
  u_\lambda^{\tau } - u_\lambda^\theta) \big) J^{s_1}_x(
  u_\lambda^\tau - u_\lambda^\theta) \, dxdy \le \\ \le & 
\frac12 \int_{\re^2} 
  \left([J^{s_1}_x , u_\lambda^{\tau } + u_\lambda^\theta] \partial_x (
  u_\lambda^{\tau } - u_\lambda^\theta) \right) J_x^{s_1} ( u_\lambda^\tau - u_\lambda^\theta) \, dxdy + \\ &+ \frac12 \int_{\re^2} 
   (u_\lambda^{\tau } + u_\lambda^\theta)   \partial_x J_x^{s_1}(
   u_\lambda^{\tau } - u_\lambda^\theta) J_x^{s_1}( u_\lambda^\tau - u_\lambda^\theta) \, dxdy \\  \le& C(\|\partial_x
   (u_\lambda^{\tau } + u_\lambda^\theta) \|_{L^\infty} \| J^{s_1-1}_x (u_\lambda^\tau -u_\lambda^\theta )\|_{L^2}\| J^{s_1}_x(u_\lambda^{\tau } - u_\lambda^\theta) \|_{L^2} +\\ &+\| J^{s_1}_x(u_\lambda^{\tau } + u_\lambda^\theta) \|_{L^2}  \|\partial_x
  (u_\lambda^\tau -u_\lambda^\theta)\|_{L^\infty} 
  \| J^{s_1}_x(u_\lambda^\tau- u_\lambda ^\theta)\|_{L^2} ) - \\ &- \frac12 \int_{\re^2} 
   \partial_x(u_\lambda^{\tau } + u_\lambda^\theta) (J^{s_1}_x ( u_\lambda^\tau -
   u_\lambda^\theta))^2 \, dxdy \\ \le & C( \|\partial_x
  (u_\lambda^{\tau } + u_\lambda^\theta) \|_{L^\infty} \|  J_x^{s_1} (u_\lambda^\tau
  -u_\lambda^\theta )\|_{L^2}^2 +\\ &+ \|\partial_x  (u_\lambda^\tau
  -u_\lambda^\theta)\|_{L^\infty} \|J_x^{s_1}(u_\lambda^{\tau } + u_\lambda^\theta)
  \|_{L^2}     \|J_x^{s_1} (u_\lambda^\tau- u_\lambda  ^\theta)\|_{L^2} ) 
\end{aligned}
\end{equation}
With the second term we proceed in the same way to obtain the
following inequality
\begin{multline}\label{eq:A1.0.2}
  \frac12 \int_{\re^2} J^{s_1}_x\big(\partial_x (u_\lambda^{\tau } +
  u_\lambda^\theta) ( u_\lambda^{\tau } - u_\lambda^\theta) \big)
  J^{s_1} \partial_x( u_\lambda^\tau - u_\lambda^\theta) \, dxdy \le
  \\ \le C ( (\|\partial_x u_\lambda^{\tau }\|_{L^\infty} +\|\partial_x
  u_\lambda^\theta \|_{L^\infty})\| J_x^{s_1} (u_\lambda^\tau
  -u_\lambda^\theta )\|_{L^2}^2 +\\ + \|J_x^{s_1}(u_\lambda^{\tau } +
  u_\lambda^\theta) \|_{L^2} \|\partial_x(u_\lambda^\tau- u_\lambda
  ^\theta)\|_{L^\infty} \|J_x^{s_1} (u_\lambda^\tau- u_\lambda
  ^\theta)\|_{L^2} )
\end{multline}
So, $\| \psi_\lambda ^\tau\|_{H^{s_1, s_2}} \le R$, for all
$\lambda$, from \eqref{eq:A1.0}, \eqref{eq:A1.0.1}, \eqref{eq:A1.0.2}
and Lemma \ref{lemaext} we have

\begin{multline}\label{eqA.1.0.C}
    \frac{1}{2} \frac{d}{dt} \| J_x^{s_1}
    (u_\lambda^\tau - u_\lambda^\theta) \|_{L^2}^2  \le C
 \big ( (\|\partial_x u_\lambda^{\tau }\|_{L^\infty} +\|\partial_x u_\lambda^\theta
  \|_{L^\infty}) \|  J_x^{s_1} (u_\lambda^\tau
  -u_\lambda^\theta )\|_{L^2}^2 + \\ +  \|\partial_x(u_\lambda^\tau- u_\lambda
  ^\theta)\|_{L^\infty} \big)
\end{multline}
 On the other hand,
\begin{equation}\label{eqA.2.0}
  \begin{aligned}
    \frac{1}{2} \frac{d}{dt} \| J_y^{s_2}
    (u_\lambda^\tau& - u_\lambda^\theta) \|_{L^2}^2 = \int_{\re^2} J_y^{s_2}
     (u_\lambda^{\tau } \partial_x
    u_\lambda^\tau - u_\lambda^{\theta } \partial_x u_\lambda^\theta) J_y^{s_2}
    ( u_\lambda^\tau - u_\lambda^\theta)
    \, dxdy \\ =&\frac12 \int_{\re^2} J^{s_2}_y \partial_x\left( (u_\lambda^{\tau }
      + u_\lambda^\theta) ( u_\lambda^{\tau } - u_\lambda^\theta) \right) J_y^{s_2}
    \partial_( u_\lambda^\tau - u_\lambda^\theta) \, dxdy \\ =& \frac12
    \int_{\re^2} J^{s_2}_y \left ((u_\lambda^{\tau } +
      u_\lambda^\tau)\partial_x ( u_\lambda^{\theta } - u_\lambda^\theta) \right)
    J^{s_2}_y ( u_\lambda^\tau - u_\lambda^\theta) \,
    dxdy + \\ &+\frac12 \int_{\re^2}  \left(
     J_y^{s_2} \partial_x( u_\lambda^{\tau } + u_\lambda^\tau) ( u_\lambda^{\theta } -
      u_\lambda^\theta) \right) J^{s_2}(
    u_\lambda^\tau - u_\lambda^\theta) \, dxdy.
  \end{aligned}
\end{equation}
Proceeding in the same way that allows us to obtain \eqref{eq:A1.0.1} and
\eqref{eq:A1.0.2}, it follows that
\begin{equation}\label{eq:B1.0.1}
  \begin{aligned}
  \frac12 \int_{\re^2} &J^{s_2}_y\big( (u_\lambda^{\tau } + u_\lambda^\theta)
  \partial_x ( u_\lambda^{\tau } - u_\lambda^\theta) \big) J^{s_2} \partial_x(
  u_\lambda^\tau - u_\lambda^\theta) \, dxdy \le\\ \le& C \big( ( \|\partial_y u_{
    n}^{\tau }\|_{L^\infty}+ \|\partial_y u_{
    n}^{\theta }\|_{L^\infty})\times \\ &\times (\| J_y^{s_2- 1}\partial_x (u_\lambda^\tau -u_\lambda^\theta
  )\|_{L^2} \| J_y^{s_2} (u_\lambda^\tau -u_\lambda^\theta )\|_{L^2} 
 + \| J_y^{s_2} (u_\lambda^\tau -u_\lambda^\theta )\|_{L^2}^2) +\\ &+ \|J_y^{s_2}(u_{
    n}^{\tau } + u_\lambda^\theta) \|_{L^2} \|\partial_x(u_\lambda^\tau- u_\lambda
  ^\theta)\|_{L^\infty} \| J_y^{s_2} (u_\lambda^\tau -u_\lambda^\theta
  )\|_{L^2} \big)  \\ \le& C\big( ( 
  \|\partial_y u_\lambda^{\tau }\|_{L^\infty} +\|\partial_y u_\lambda^\theta
  \|_{L^\infty}) (\| J_x^{s_1} (u_\lambda^\tau -u_\lambda^\theta )\|_{L^2}^2 +\|
  J_y^{s_2} (u_\lambda^\tau -u_\lambda^\theta )\|_{L^2}^2) +\\ &+  \|\partial_x(u_\lambda^\tau- u_\lambda
  ^\theta)\|_{L^\infty} \big)
\end{aligned}
\end{equation}
and
\begin{equation}\label{eq:B1.0.2}
  \begin{aligned}
  \frac12 \int_{\re^2} 
  J^{s_2}_y\big( & \partial_x (u_\lambda^{\tau } + u_\lambda^\theta)(
  u_\lambda^{\tau } - u_\lambda^\theta) \big) J^{s_2}  \partial_x( u_\lambda^\tau
  - u_\lambda^\theta) \, dxdy \le \\ \le&  C\big(
  ( \|\partial_x u_\lambda^{\tau }\|_{L^\infty} +\|\partial_x u_\lambda^\theta
  \|_{L^\infty} )\|  J_y^{s_2} (u_\lambda^\tau
  -u_\lambda^\theta )\|_{L^2}^2 +\\ &+ \|J_y^{s_2-1} \partial_x ( u_{
    n}^{\tau } + u_\lambda^\theta)  
  \|_{L^2}     \|J_y^{s_2} (u_\lambda^\tau- u_\lambda
  ^\theta)\|_{L^2} \|\partial_y (u_\lambda^\tau- u_\lambda
  ^\theta)\|_{L^\infty } )\\ \le& C
  ( \|\partial_x u_\lambda^{\tau }\|_{L^\infty} +\|\partial_x u_\lambda^\theta
  \|_{L^\infty} )( \| J_x^{s_1} (u_\lambda^\tau -u_\lambda^\theta )\|_{L^2}^2 +\|  J_y^{s_2} (u_\lambda^\tau
  -u_\lambda^\theta )\|_{L^2}^2) +\\ &+ \|\partial_y (u_\lambda^\tau- u_\lambda
  ^\theta)\|_{L^\infty })
\end{aligned}
\end{equation}
Therefore, from \eqref{eqA.2.0}, \eqref{eq:B1.0.1} and
\eqref{eq:B1.0.2}, we have

\begin{multline}\label{eqB.1.0.C}
    \frac{1}{2} \frac{d}{dt} \| J_y^{s_2}
    (u_\lambda^\tau - u_\lambda^\theta) \|_{L^2}^2 \le \\ \le C
  ( \|\partial_x u_\lambda^{\tau }\|_{L^\infty} +\|\partial_x u_\lambda^\theta
  \|_{L^\infty} + \|\partial_y u_\lambda^{\tau }\|_{L^\infty} +\|\partial_y u_\lambda^\theta
  \|_{L^\infty})\times \\ \times( \| J_x^{s_1} (u_\lambda^\tau -u_\lambda^\theta )\|_{L^2}^2 +\|  J_y^{s_2} (u_\lambda^\tau
  -u_\lambda^\theta )\|_{L^2}^2) + \\ + \|\partial_x (u_\lambda^\tau- u_\lambda
  ^\theta)\|_{L^\infty }+ \|\partial_y (u_\lambda^\tau- u_\lambda
  ^\theta)\|_{L^\infty })
\end{multline}
\par Gathering \eqref{eqA.1.0.C} and \eqref{eqB.1.0.C}, we obtain
\begin{multline}\label{eqC.1.0.C}
    \frac{1}{2} \frac{d}{dt} \| u_\lambda^\tau - u_\lambda^\theta
    \|_{H^{s_1, s_2}}^2 \le \\ \le C \big(
  ( \|\partial_x u_\lambda^{\tau }\|_{L^\infty} +\|\partial_x u_\lambda^\theta
  \|_{L^\infty} + \|\partial_y u_\lambda^{\tau }\|_{L^\infty} +\|\partial_y u_\lambda^\theta
  \|_{L^\infty}) \| u_\lambda^\tau -u_\lambda^\theta
  \|_{H^{s_1, s_2}}^2 +\\  + \|\partial_x (u_\lambda^\tau- u_\lambda
  ^\theta)\|_{L^\infty }+ \|\partial_y (u_\lambda^\tau- u_\lambda
  ^\theta)\|_{L^\infty }\big)
\end{multline}
Integrating, it follows that
\begin{multline} \label{eq:2.55}
     \| u_\lambda^\tau - u_\lambda^\theta \|_{H^{s_1, s_2}}^2   \le \|
     \psi_\lambda^\tau - \psi_\lambda^\theta\|_{H^{s_1, s_2}}^2 +\\ + C
 \int_0^t  (\|\partial_x u_\lambda^{\tau }\|_{L^\infty}
 +\|\partial_x u_\lambda^\theta 
  \|_{L^\infty}) \|  u_\lambda^\tau
  -u_\lambda^\theta \|_{H^{s_1, s_2}}^2 \, dt' +\\ +   \|\partial_x(u_\lambda^\tau- u_\lambda
  ^\theta)\|_{L_{T_0}^1 L_{x,y}^\infty} +\|\partial_x(u_\lambda^\tau- u_\lambda
  ^\theta)\|_{L_{T_0}^1 L_{x,y}^\infty}
\end{multline}
Let us estimate the last two  terms of the last inequality. Observe that
$$ u_\lambda^\tau- u_\lambda ^\theta = W_\alpha(t) (\psi_\lambda^\tau-
\psi_\lambda ^\theta) + \int_0^t W_\alpha(t-t') \partial_x
((u_\lambda^\tau+ u_\lambda ^\theta) (u_\lambda^\tau- u_\lambda
^\theta))(t') \, dt'. $$ Therefore, proceeding as in the
proof of Lemma \ref{lemaestnorstr}, from refined the Strichartz
estimate \eqref{eqstrichartzref} and Lemma \ref{lem:L2lip}, we have
  $$
  \begin{aligned}
  \|\partial_x( u_\lambda^\tau- u_\lambda
  ^\theta)\|_{L_{T_0}^1 L_{x,y}^\infty} \le& C (\|u_\lambda^\tau-
  u_\lambda ^\theta \|_{L_{T_0}^\infty H_{x,y}^{s_1', s_2'}} +
  \|(u_\lambda^\tau +
  u_\lambda ^\theta )(u_\lambda^\tau-
  u_\lambda ^\theta) \|_{L_{T_0}^\infty H_{x,y}^{s_1', s_2'}})\\
  \le& C (\|u_\lambda^\tau-
  u_\lambda ^\theta \|_{L_{T_0}^\infty H_{x,y}^{s_1', s_2'}} + \|u_\lambda^\tau+
  u_\lambda ^\theta \|_{L_{T_0}^\infty H_{x,y}^{s_1', s_2'}}\| u_\lambda^\tau-
  u_\lambda ^\theta) \|_{L_{T_0}^\infty H_{x,y}^{s_1', s_2'}})\\ \le&
  C\|u_\lambda^\tau- 
  u_\lambda ^\theta \|_{L_{T_0}^\infty H_{x,y}^{s_1', s_2'}} \\ \le &
  C \|\psi_\lambda^\tau- \psi_\lambda ^\theta\|_{L^2}^{\nu}  \|u_\lambda^\tau-
  u_\lambda ^\theta \|_{L_{T_0}^\infty H_{x,y}^{s_1, s_2}}^{1-\nu} \\ \le &
  C \tau^\nu,  
\end{aligned}
$$
where $s'_1$, $s'_2$ and $\nu$ are such that $ 1< s'_i =\nu s_i < s_i$,
$i=1,2$. In the same way, again, if we proceed as in the proof
of Lemma \ref{lemaestnorstr}, more precisely, in the way that we got 
\eqref{eqestnorstr15}, we also have
$$ \|\partial_y( u_\lambda^\tau- u_\lambda ^\theta)\|_{L_{T_0}^1
  L_{x,y}^\infty} \le C\tau^\nu $$ From these last two inequalities
with \eqref{eq:2.55} and the Gronwall Lemma follows the proposition.
\end{proof}

\begin{corol} \label{prop11a} Let $\psi_\lambda^\tau$ and
  $u_{\lambda}^\tau$ be as in the result above. Then,
  $ \left\lbrace u_{\lambda}^\tau \right\rbrace _{\tau > 0}$ converges
  uniformly in $t$ to $u_{\lambda}$, when $\tau \to 0+$. In other
  words,
  \[ \lim_{\tau \to 0+} \sup_{t\in [0,T]}
    \ns{u_{\lambda}^\tau(t)-u_{\lambda}(t)}{H^{s_1, s_2}}=0. \]
\end{corol}
\begin{proof}
  Take $\theta = 0$ in Proposition \ref{prop10}.
\end{proof}
\begin{teo}
  In $H^{s_1, s_2}$ the map $\psi \mapsto u$, where $u$ is solution to
  \eqref{eq:principal} with initial condition $\psi$, is continuous. More
  precisely, if $(\psi_n)_{n\in \na}$ is a sequence such that
  $\psi_n \to \psi$ in $H^{s_1, s_2}$ and if
  $u_{n} \in C([0,T_0];H^{s_1, s_2})$ are the corresponding solutions
  to \eqref{eq:principal} with initial condition $\psi_n$, then
\[ \lim_{n \to \infty}\sup_{t\in [0,T_0]} \ns{u_{n}(t)- u(t)}{s}=0 \]
\end{teo}

\begin{proof}
  Let $\psi\in H^{s_1, s_2}$ and $(\psi_n)$ be a sequence in
  $H^{s_1, s_2}$ that converges strongly to $\psi$ in this
  space. Let, also,
  $R=\max (\sup_n \| \psi_n\|_{H^{s_1, s_2}} ,\| \psi\|_{H^{s_1,
      s_2}})$.  Now, let us take
  $\psi_{m,n}= e^{\frac 1m \bigtriangleup } \psi_n$ and
  $\psi_{m}= e^{\frac 1m \bigtriangleup } \psi$.  Let $u_{n}$, $u$,
  $u_{m,n}$, $u_{m}$, $u^\tau_{n}$, $u^\tau$, $u^\tau_{m,n}$,
  $u^\tau_{m}$ be the corresponding solutions to \eqref{eq:principal}
  in $[0, T_0]$ with conditions  $\psi_n$, $\psi$,
  $\psi_{m,n}$, $\psi_{m}$, $\psi^\tau_{n}$, $\psi^\tau$,
  $\psi^\tau_{m,n}$, $\psi^\tau_{m}$, respectively. Also,
  observe that $u_{n,m}$, $u_m$, $u^\tau_{m,n}$ and $u^\tau_{m}$
  converge uniformly to $u_n$, $u$, $u^\tau_{n}$ and $u^\tau$ in the
  weak sense, as $m\to \infty$.  Therefore,
\begin{align*}
  \langle u_{n}-u, \varphi\rangle_{H^{s_1, s_2}} = & \lim_{m\to \infty} \langle
  u_{m,n}-u_{m,n}^\tau, \varphi\rangle_{H^{s_1, s_2}} +
  \langle u_{m,n}^\tau-u_{m}^\tau, \varphi\rangle_{H^{s_1, s_2}}+\\ & +
  \langle u_{m}^\tau-u_m, \varphi\rangle_{H^{s_1, s_2}}  \\ = \lim_{m
    \to \infty} &\left[ \langle u_{m,n}-u_{m,n}^\tau, \varphi\rangle_{H^{s_1, s_2}} +
    \langle u_{m,n}^\tau-u_{m}, \varphi\rangle_{H^{s_1, s_2}} \right]+\\ & +
  \langle u_{n}^\tau-u^\tau, \varphi\rangle_{H^{s_1, s_2}}.
\end{align*}
On the other hand, Corollary \ref{prop11a} implies that, given $\epsilon
>0$, there exists $\tau_0$ such that for $0<\tau\le \tau_0$
\[ \lvert \langle u_{m,n}-u_{m,n}^\tau, \varphi\rangle_{H^{s_1, s_2}} +
\langle u_{m,n}^\tau-u_{m}, \varphi\rangle_{H^{s_1, s_2}} \rvert \leq
\epsilon \ns{\varphi}{H^{s_1, s_2}}, \] for all $m > 0$.  Then,
\begin{eqnarray*}
  \lvert \langle u_{n}-u, \varphi\rangle_{H^{s_1, s_2}} \rvert
  \leq \epsilon \ns{\varphi}{H^{s_1, s_2}} +
  \ns{u_{n}^\tau-u^\tau}{H^{s_1, s_2}} \ns{\varphi}{H^{s_1, s_2}} ,
\end{eqnarray*}
for all $\varphi \in H^{s_1, s_2}$.
Therefore, 
\begin{equation}\label{desetau} 
\ns{u_{n}-u}{H^{s_1, s_2}} \leq \epsilon +
\ns{u_{n}^\tau-u^\tau}{H^{s_1, s_2}}. 
\end{equation}
\par
Arguments similar to those used in Proposition \ref{prop10}
allow us show that, for $\tau$ small enough,
$$\ns{u_{n}^\tau-u^\tau}{H^{s_1, s_2}} \leq
C\ns{\psi_{n}^\tau-\psi^\tau}{H^{s_1, s_2}} \tau^{-\frac{1}{s}} \leq
\ns{\psi_{n}- \psi }{H^{s_1, s_2}} \tau^{-\frac{1}{s}}.$$ Then,
fixing $\tau$ small enough, we can conclude from
\eqref{desetau} that
\[ \ns{u_{n}-u}{H^{s_1, s_2}} \leq 2\epsilon , \]
for $n$ large enough. 
\end{proof}
Let us see now that the problem is locally well-posed in the spaces
$X^{s_1, s_2}$, $\widehat X^{s_1, s_2}$ and $Y^{s_1, s_2}$. Since the
solution to \eqref{eq:principal} satisfy the integral equation
$$ u= W_\alpha(t)\psi + \int_0^t W_\alpha(t-t') (u \partial_x u)
(t') \, dt', $$ we have that $\partial_x^{-1} u$ and
$\partial_x^{-1}\partial_y u$ satisfy the equations  
$$ \partial_x^{-1} u= W_\alpha(t)\partial_x^{-1} \psi + \int_0^t
W_\alpha(t-t')\left( \frac {u^2} 2 \right)
(t') \, dt'$$
and
$$ \partial_x^{-1}\partial_y  u= W_\alpha(t) \partial_x^{-1}\partial_y
\psi + \int_0^t W_\alpha(t-t') (u \partial_y u) (t') \, dt', $$
respectively. From here and Theorem \ref{bplxs}, for $s_1$ and $s_2$
as in that theorem, it follows that \eqref{eq:principal} is locally
well-posed in the spaces $\widehat X^{s_1, s_2}$, $ X^{s_1, s_2}$
and $\widehat Y^{s_1, s_2}$. The case $Y^{s_1, s_2}$ requires some
additional effort.  For $\psi \in Y^\infty$, the solution $u$ to
\eqref{eq:principal}, with initial condition $\psi$, belongs to
$C([0, T_0]; Y^\infty)$. We can proceed as in the proof Lemma
\ref{lemaestenergia}, to show that
\begin{equation*}
   \frac12\frac d{dt} \|\partial_x^{-1}\partial_yu\|^2 \le C(\|u_x\|+ \|u_y\|)
  \|u\|_{Y^{s_1, s_2}}^2 .
\end{equation*}
This inequality with \eqref{eqestenergia16} allow us prove that
\begin{equation}
  \label{estenergia17}
  \frac12\frac d{dt} \|u\|_{Y^{s_1,s_2}}^2 \le C(\|u_x\|_\infty+ \|u_y\|_\infty)
  \|u\|_{Y^{s_1, s_2}}^2 .
\end{equation}
Thanks to the Gronwall lemma we have the following generalization of
Lemma  \ref{lemaestenergia},
$$
\|u\|_{Y^s}\le \|\psi\|_{Y^{s_1, s_2}} e^{C(\|u_x\|_{L^1_t L_{x,y}^\infty}+
  \|u_y\|_{L^1_t L_{x,y}^ \infty})}.
$$
Now, if $\psi$ is arbitrary, in the same way that we prove for the
space $H^{s_1,s_2}$, the solution $u\in C([0,T];Y^{s_1, s_2})$. To see
that the map $\psi\mapsto u$ is continuous from ${Y^{s_1, s_2}}$ to
$C([0,T];Y^{s_1, s_2})$ we repeat the same argument of Bona-Smith that
we use before. So, summarizing, we paraphrase Theorem \eqref{teo:2.7}
for the current situation.
\begin{teo}
  Let $s_1, s_2$ and $\alpha$ be as in Theorem \ref{bplxs}. Let,
  also, $Z$ any of the spaces $X^{s_1, s_2}$, $\widehat X^{s_1, s_2}$,
  $Y^{s_1, s_2}$ and $\widehat Y^{s_1, s_2}$. Then, if $\psi\in Z$ and
  $u\in C([0,T]; H^{s_1, s_2})$ is solution to \eqref{eq:principal} with
  $u(0)=\psi$, then $u\in C([0,T]; Z)$. Moreover,
  $\psi \mapsto u$ is continuous from $Z$ to $C([0,T]; Z)$
\end{teo}

%%% Local Variables:
%%% mode: latex
%%% TeX-master: "allart"
%%% End:

%% file: Factart.tex
\section{Remarks on ill-posedness of the equation
  \eqref{eq:principal}}
In this section we prove that the flow associated to the equation
\eqref{eq:principal} is not of class $C^2$ for $-1\le \alpha <0$. In
particular, we have that we cannot apply Picard iterative process to
solve the Duhamel equation associated to this equation.\par
For this we will use the ideas given in \cite{MST} to prove that
the flow associated with the KP-I equation is not of class $C^2$.
\subsection{The flow associated to the problem \eqref{eq:principal} is  not $C^2$}
\begin{teo}
  Let $(s_1,s_2)\in \re^2$ and $-1 \le \alpha<0$. Then, there does not
  exist $T>0$ such that \eqref{eq:principal} has a unique solution
  $u$ for all $\phi \in H^{s_1.s_2}$ and that the flow
  $S_t:\phi \mapsto u$ is not of class $C^2$ at $0$ from $H^{s_1.s_2}$ to
  $H^{s_1.s_2}$
\end{teo}
\begin{proof}
 Let's see first what we should show. For this we consider
  $u(\lambda,t)= S_t (\lambda\phi)$, $S_t$ the flow associated to
  the problem \eqref{eq:principal}. So this solution satisfies
  the Duhamel equation associated to the equation
  \eqref{eq:principal}, i. e.,
   \begin {equation} \label{eq:4}
    u(\lambda, t)= \lambda W_\alpha(t)\phi -\int_0^t W_\alpha(t-t') u(t') u_x(t')\,
    dt'.
  \end{equation} 
  If the flow is twice differentiable around $ 0 $ in
  $H^{s_1.s_2}$ then, thanks to the chain rule,
  $$ \partial_\lambda u(0,t)= W_\alpha(t)\phi,$$ and $$ \partial_\lambda^2 u(0,t) = 
  -2\int_0^t W_\alpha(t-t') W_\alpha(t')\phi W_\alpha(t')\phi_x\,
  dt'. $$ Which would imply that the map
  $$\phi \mapsto \int_0^t W_\alpha(t-t') W_\alpha(t')\phi W_\alpha(t')\phi_x \, dt' $$
 would be a quadratic form coming from a continuous symmetric
 bilinear transformation in $ H ^ {s_1.s_2} $, and that, in
 particular, for some fixed $ C $ 
  $$ \left\| \int_0^t W_\alpha(t-t') W_\alpha(t')\phi W_\alpha(t')\phi_x\,
    dt' \right\|_{H^{s_1.s_2}} \le C\|\phi\|_{H^{s_1.s_2}} ^2 $$ for
  all $\phi \in H^{s_1.s_2}$. So let's show that this is not the
  case. To do this, suppose that this inequality is valid and consider
  the function $\phi$ defined via the Fourier transform by
  $$ \widehat \phi = \gamma^{-3/2} \mathbbm 1_{D_1} +
  \gamma^{-3/2}N^{-s_1 - \frac {3 -\alpha} 2s_2} \mathbbm 1_{D_2} , $$
  where $\gamma$ and $N$ are positive numbers such that $\gamma \ll 1$
  and $N\gg 1$, $D_1$ and $D_2$ are the sets
  $$ D_1=\left[ \frac\gamma 2, \gamma \right]\times \left[ -\frac
    {\gamma^2} 6, \frac {\gamma^2} 6 \right] \ \text{and} \ D_2=\left[ N,
    N+\gamma \right]\times \left[ \sqrt{-\frac 3 {\alpha}} N^{\frac {3 -\alpha} 2},
    \sqrt{-\frac 3 {\alpha}} N^{\frac {3 -\alpha} 2}+\gamma^2 \right]$$ and $\mathbbm 1_{D_i}$
  are the characteristic functions of the sets $D_i$,
  $i=1,2$. $\|\phi\|_{H^{s_1. s_2}}\sim 1$  for whatever values that
  we take for $\gamma$ and $N$. Let us see that for a convenient
  choice in parameters of $\gamma$ and $N$, \begin{equation}\label{eq:5} \left\| \int_0^t
    W_\alpha(t-t') W_\alpha(t')\phi W_\alpha(t')\phi_x\, dt' \right\|_{H^{s_1.s_2}}
\end{equation}
can be as large as we want. Calculating its Fourier transform, it
follows that  \begin{equation}\label{eq:6} \int_0^t W_\alpha(t-t') W_\alpha(t')\phi
  W_\alpha(t')\phi_x\, dt'
  \end{equation}
  is $f_1+f_2+f_3$ where $$
    \widehat f_1(t, \xi,\eta)= \frac{ i\xi e^{it( \xi^3 + \sgn(\xi)
         |\xi|^{\alpha} \eta^2)}} {2\gamma^3} \int_{\substack{ (\xi_1, \eta_1) \in
         D_1\\ (\xi-\xi_1, \eta- \eta_1) \in D_1}} \frac
       {e^{-it\chi(\xi, \xi_1, \eta, \eta_1)} -1} {\chi(\xi, \xi_1,
         \eta, \eta_1) } \, d\xi_1d\eta_1,$$  $$
    \widehat f_2(t, \xi,\eta)= \frac{ i\xi e^{it( \xi^3 + \sgn(\xi)
         |\xi|^{\alpha} \eta^2)}} {2\gamma^3 N^{ 2(s_1 + \frac{3
           -\alpha} 2 s_2)}} \int_{\substack{ (\xi_1, \eta_1) \in
         D_2\\ (\xi-\xi_1, \eta- \eta_1) \in D_2}} \frac
       {e^{-it\chi(\xi, \xi_1, \eta, \eta_1)} -1} {\chi(\xi, \xi_1,
         \eta, \eta_1) } \, d\xi_1d\eta_1$$  and $$
       \begin{aligned}
    \widehat f_3(t, \xi,\eta)=& \frac{ i\xi e^{it( \xi^3 + \sgn(\xi)
         |\xi|^{\alpha} \eta^2)}} {2\gamma^3 N^{ s_1 + \frac{3
           -\alpha} 2 s_2}} \int_{\substack{ (\xi_1, \eta_1) \in
         D_2\\ (\xi-\xi_1, \eta- \eta_1) \in D_1}} \frac
       {e^{-it\chi(\xi, \xi_1, \eta, \eta_1)} -1} {\chi(\xi, \xi_1,
         \eta, \eta_1) } \, d\xi_1d\eta_1 +\\ &+ \frac{ i\xi e^{it(
           \xi^3 + \sgn(\xi) 
         |\xi|^{\alpha} \eta^2)}} {2\gamma^3 N^{ s_1 + \frac{3
           -\alpha} 2 s_2}} \int_{\substack{ (\xi_1, \eta_1) \in
         D_1\\ (\xi-\xi_1, \eta- \eta_1) \in D_2}} \frac
       {e^{-it\chi(\xi, \xi_1, \eta, \eta_1)} -1} {\chi(\xi, \xi_1,
         \eta, \eta_1) } \, d\xi_1d\eta_1 ,
     \end{aligned}
     $$   where $\chi$ is the resonant function  \begin{equation}
\begin{split}
    \chi&= \chi(\xi, \xi_1, \eta, \eta_1)=\vartheta(\xi,\eta)-\vartheta(\xi_1,\eta_1)-\vartheta(\xi-\xi_1,\eta-\eta_1)\\
    &=3\xi\xi_1(\xi-\xi_1)+\sgn(\xi) \frac {\eta^{2}}
    {|\xi|^{\theta}}-\sgn(\xi_1) \frac {\eta_1^{2}}
    {|\xi_1|^{\theta}}-\sgn(\xi-\xi_1) \frac {(\eta-\eta_1)^{2}}
    {|\xi-\xi_1|^{\theta}}, 
\end{split}
\end{equation}
where
\begin{equation}
    \vartheta(\xi,\eta)=\xi^3+\sgn(\xi) \frac {\eta^{2}} {|\xi|^{\theta}},
\end{equation}
is the phase function and $\theta=-\alpha$, which is between $0$ and $1$.
Since in our case $\xi, \xi_1$ and $\xi-\xi_1$  are  positive,
we have
 \begin{equation}
    \chi=3\xi\xi_1(\xi-\xi_1)+ \frac {\eta^{2}}{\xi^{\theta}}
    - \frac {\eta_1^{2}}{ \xi_1^{\theta}}- \frac {(\eta-\eta_1)^{2}}
    {(\xi-\xi_1)^{\theta}} 
  \end{equation}
  Observe that to estimate \eqref{eq:5} it is enough estimate
  $\|f_3(t)\|_{H^{s_1, s_2}}$, in fact, $$ \left\| \int_0^t W_\alpha(t-t')
    W_\alpha(t')\phi W_\alpha(t')\phi_x\, dt' \right\|_{H^{s_1.s_2}}\ge
  \|f_3(t)\|_{H^{s_1, s_2}}  .$$  To continue we need the following lemma.
  \begin{lema}
    Suppose that
    $$(\xi_1, \eta_1)\in D_1\qquad \text{and} \qquad (\xi -\xi_1, \eta-
    \eta_1)\in D_2 $$ or
    $$(\xi_1, \eta_1)\in D_2\qquad \text{and} \qquad (\xi -\xi_1, \eta-
    \eta_1)\in D_1. $$ Then,
    $$|\chi(\xi, \xi_1, \eta, \eta_1)| \lesssim \gamma^2 N.$$
  \end{lema}
  \begin{proof}
First let us calculate the $\eta$ such that $\chi(\xi, \xi_1, \eta,
\eta_1)=0$. So, we have
 \begin{equation}
   [\xi^\theta-(\xi-\xi_1)^{\theta}]\xi_1^\theta\eta^2-2\xi^\theta\xi_1^\theta\eta_1\eta+[\xi_1^\theta+(\xi-\xi_1)^{\theta}]\xi^\theta\eta_1^2-3\xi^{1+\theta}\xi_1^{1+\theta}(\xi-\xi_1)^{1+\theta}=0
 \end{equation}
 Then, we get 
 \begin{equation}
         \eta
        =\frac{\xi^{\theta}\eta_1}{[\xi^{\theta}-(\xi-\xi_1)^{\theta}]}\pm
        \sqrt{\frac{3\xi^{1+\theta}\xi_1(\xi-
            \xi_1)^{1+\theta}}{[\xi^\theta-(\xi-\xi_1)^{\theta}]}+
          \frac{\xi^{\theta}(\xi -\xi_1)^{\theta}[\xi_1^\theta+
            (\xi-\xi_1)^{\theta}-\xi^\theta]\eta_1^2}{[\xi^\theta-(\xi-\xi_1
             )^{\theta}]^2\xi_1^{\theta}}}
 \end{equation}
Let $\eta*$ be the smallest zero between the two calculated above. So,
 \begin{equation}
  \eta^*-\eta_1=\frac{(\xi-\xi_1)^{\theta}\eta_1}{[\xi^{\theta}-(\xi-\xi_1
    )^{\theta}]}-
  \sqrt{\frac{3\xi^{1+\theta}\xi_1(\xi-\xi_1)^{1+\theta} }{
      [\xi^\theta-(\xi- \xi_1)^{\theta}]}+\frac{\xi^{\theta}(\xi-
      \xi_1 )^{\theta}[\xi_1^\theta+(\xi-\xi_1)^{\theta}-\xi^\theta]
      \eta_1^2 }{[\xi^\theta-(\xi-\xi_1)^{\theta}]^2\xi_1^{\theta}}}.
  \end{equation}
  Let 
 \begin{equation*} 
%\begin{split}
R  = \frac{(\xi-\xi_1)^{\theta}\eta_1}{[\xi^{\theta}-(\xi-\xi_1
    )^{\theta}]}+
  \sqrt{\frac{3\xi^{1+\theta}\xi_1(\xi-\xi_1)^{1+\theta} }{
      [\xi^\theta-(\xi- \xi_1)^{\theta}]}+\frac{\xi^{\theta}(\xi-
      \xi_1 )^{\theta}[\xi_1^\theta+(\xi-\xi_1)^{\theta}-\xi^\theta]
      \eta_1^2 }{[\xi^\theta-(\xi-\xi_1)^{\theta}]^2\xi_1^{\theta}}}.
%\end{split}
\end{equation*}
Then
\begin{equation}
  |\eta^*-\eta_1|=\dfrac{ \frac{(\xi^{\theta}-\xi_1^\theta)(\xi-
      \xi_1 )^{\theta}}{
      (\xi^\theta-(\xi-\xi_1)^{\theta})\xi_1^{\theta}} 
      {\left|\eta_1^2 - 3\xi^{1+\theta}\xi_1^{1+\theta} g(\xi, \xi_1)
        \right|}} R,
\end{equation}
where $$g(\xi,\xi_1)= \begin{cases} \dfrac {\xi -\xi_1} {\xi^\theta
    -\xi_1^\theta} & \text{ if } \xi\ne \xi_1 \\[5mm] \frac1\theta
  \xi_1^{1-\theta} & \text{ in another case.}
\end{cases} 
$$
Now, since \begin{equation} \label{eq:1} \frac 1\theta \xi_1^{1-\theta} \le g(\xi,\xi_1)\le
\frac 1\theta \xi^{1-\theta}
\end{equation}
and $$ R\ge \frac{(\xi-\xi_1)^{ \theta}
  \eta_1}{ (\xi^{ \theta}-(\xi-\xi_1 )^{\theta})},
$$
we have
$$
\begin{aligned}
  |\eta^*-\eta_1|&= \frac{\xi^{\theta}-\xi_1^\theta}{\xi_1^{\theta}} 
      \frac {\left|\eta_1^2 - 3\xi^{1+\theta}\xi_1^{1+\theta} g(\xi, \xi_1)
        \right|} {\eta_1} \\ &\le
      3\frac{\xi^{\theta}-\xi_1^\theta}{\xi_1^{\theta}} \left|\eta_1 -
        \sqrt {3}\xi^{\frac{1+\theta} 2}\xi_1^{\frac {1+\theta} 2}
         \sqrt{g(\xi, \xi_1)} 
        \right|  \\ &\le
      3\theta \frac{\xi-\xi_1}{\xi_1} \left|\eta_1 -
        \sqrt{\frac  {3}\theta} \xi_1^{\frac{3+\theta} 2}- \xi_1^{\frac
          {1+\theta} 2} 
         \left( \xi^{\frac {1+\theta} 2}\sqrt{g(\xi, \xi_1)} -\frac 1
         {\sqrt{\theta} }\xi_1 \right) \right|  ,
    \end{aligned}
    $$
Remember that $\eta_1$ take values in $\left[\sqrt{\frac 3\theta
}N^{ \frac{ 3+\theta} 2}, \sqrt{\frac 3 \theta}N^{\frac{ 3+\theta} 2}+
\gamma^2\right]$ and that $\xi_1 $ in $[N,
N+\gamma]$. Whence, thanks to the Taylor formula with remainder
applied to the function $x\mapsto x^{\frac{ 3+\theta}2}$, $$\sqrt{3}\xi_1^{\frac{ 3+\theta} 2} \in
\left[\sqrt{3}N^{\frac{ 3+\theta} 2}, \sqrt{3}N^{\frac{ 3+\theta}2} +
\sqrt{3} \frac{ 3+\theta}2 N^{\frac{ 1+\theta}2} \gamma +
\sqrt{3}\frac{ (3+\theta)(1+\theta)}8  \gamma^2 \right] $$ and, moreover $$ 
\left| \eta_1- \sqrt{3} \xi_1^{\frac{ 3+\theta}2 } \right|\le
2\sqrt{3} N^{\frac{1+\theta}2} \gamma +\sqrt{3} \gamma^2. $$ On the
other hand from \eqref{eq:1} $$\xi^{\frac {1+\theta} 2}\sqrt {g(\xi, \xi_1)}
-\frac1 {\sqrt{\theta}} \xi_1\le \frac1 {\sqrt{\theta}}(\xi-\xi_1)\le
\frac1 {\sqrt{\theta}} \gamma.$$ Therefore,
$$ |\eta^* -\eta_1|\le 3\theta \frac\gamma N \left(2\sqrt{\frac 3\theta} N^
  {\frac{1+ \theta}2}
\gamma +\sqrt{\frac 3 \theta} \gamma^2+
\sqrt{\frac 3 \theta}(N^{\frac{1+\theta}2}+\gamma)\gamma \right)\le
18\sqrt{\theta} \frac{ \gamma^2} {N^{\frac{1 -\theta} 2}}.   $$
By the mean value theorem, there exists $\bar \eta\in [\eta, \eta^*]$ such
that $$
\begin{aligned}
  \chi(\xi,\xi_1,\eta,\eta_1)& = \chi(\xi,\xi_1,\eta^*,\eta_1)
+(\eta- \eta^*) \partial_\eta \chi(\xi,\xi_1,\bar\eta,\eta_1) \\ &=
-(\eta- \eta^*)\frac {2(\bar \eta(\xi^\theta - (\xi- \xi_1)^\theta) -\eta_1
 \xi^\theta)} {\xi^\theta (\xi- \xi_1)^\theta} \\ &=
-(\eta- \eta^*)\frac {2((\bar \eta- \eta_1)\xi^\theta -(\xi-
  \xi_1)^\theta \bar \eta) 
} {\xi^\theta (\xi- \xi_1)^\theta} \\ & =
-(\eta- \eta^*)2\left( \frac {\bar \eta- \eta_1} {(\xi- \xi_1)^\theta}
  - \frac  {\bar \eta} {\xi^\theta}\right).
\end{aligned}
$$
So,
$$
\begin{aligned}
|\chi(\xi,\xi_1,\eta,\eta_1)| &\lesssim \frac{ \gamma^2}
{N^{\frac{1 -\theta} 2}} \left( \frac{ \gamma^2} {N^{\frac{1 -\theta}
      2} \gamma^\theta} + \frac{ N^{ \frac {3+\theta} 2}} {N^\theta}
\right )\\ &\lesssim \gamma^2 N
\end{aligned}
$$
The lemma follows immediately observing that
$$ \chi(\xi,\xi_1,\eta,\eta_1)=\chi(\xi,\xi-\xi_1,\eta,\eta-\eta_1)
.$$
\end{proof}
Let us finish the proof of the theorem. Let us choose $\gamma$ and $N$
in such a way that $\gamma^2N=N^{-\varepsilon}$ for $\varepsilon\ll
1$. Thanks to the previous lemma we have that
$$ \left| \frac {e^{it\xi} -1} {\xi} \right | = |t|+
O(N^{-\epsilon})$$ for $(\xi_1, \eta_1) \in D_1$ and
$(\xi- \xi_1,\eta- \eta_1) \in D_2$ or $(\xi_1, \eta_1) \in D_2$ and
$(\xi- \xi_1,\eta- \eta_1) \in D_1$. So
$$ \|f_3(t,\cdot, \cdot)\|_{H^s} \gtrsim \frac{NN^{\frac {3
      +\theta}2s} \gamma^3 \gamma^\frac32} { N^{\frac {3 +\theta} 2 s}
  \gamma^3}=\gamma^{\frac32}N= N^{(1-3\varepsilon)/4}.$$ This leads to a
contradiction since
$$ 1\sim \|\phi\|_{H^s}^2 \gtrsim \|f_3(t,\cdot, \cdot)\|_{H^s}.$$
\end{proof}
An immediate consequence of the previous theorem is the following
theorem.
\begin{teo}
  For $(s_1,s_2)\in \re^2$, $-1\le \alpha <0$ and a positive real
  number $T$, there does not exists a space $X_T$ continuously embedded in
  $C([-T, T]; H^{s_1,s_2})$ such that, for a fix constant
  $C$, \begin{equation} \label{eq:2} \|W_\alpha(\cdot) \phi\|_{X_T}
    \le C\|\phi\|_{H^{s_1,s_2}},
  \end{equation}
  for all $\phi\in H^{s_1,s_2}$, and
  \begin{equation} \label{eq:3}
    \left\| \int_0^t W_\alpha(t-t') (u(t')u_x(t') ) \, dt' \right\|_{X_T}
    \le C\|u\|_{X_T}^2,
  \end{equation}
  for all $u\in X_T$.
\end{teo}
Note that the estimates given in the theorem statement are necessary
to prove the contraction properties of the operator 
$\Phi$ defined by $$ \Phi(u)= W_\alpha(t) \phi + \int_0^t W_\alpha(t-t')
(u(t')u_x(t') ) \, dt'.$$
\begin{proof}
Suppose that we have \eqref{eq:2} and \eqref{eq:3} for all $\phi\in
H^{s_1,s_2}$ and for all $u\in X_T$. Let $\phi\in
H^{s_1,s_2}$ and set $u(t)= W_\alpha(t)\phi$. Then, 
\begin{equation*} 
    \left\| \int_0^t W_\alpha(t-t') (W_\alpha(t')\phi  W_\alpha(t')\phi_x  ) \,
      dt' \right\|_{X_T}  \le C\|W_\alpha(t)\phi\|_{X_T}^2\le \|\phi\|_{
      H^{s_1,s_2}}^2.
  \end{equation*}
  Since $X_T$ is continuously embedded in $C([-T, T];
  H^{s_1,s_2})$
  $$   \left\| \int_0^t W_\alpha(t-t') (W_\alpha(t')\phi  W_\alpha(t')\phi_x  ) \,
      dt' \right\|_{H^{s_1,s_2}}  \le C\|\phi\|_{H^{s_1,s_2}}^2,
    $$
    which is contradictory with the previous theorem. This ends the proof.
  \end{proof}
  Another immediate corollary is the following theorem.
  \begin{teo}
    The flow associated to the problem \eqref{eq:principal}, for
    $-1\le \alpha<0$, whose well-posedness was proved in the previous
    section, is not of class $C^2$.
  \end{teo}

%%% Local Variables:
%%% mode: latex
%%% TeX-master: "allart"
%%% End: